\documentclass[a4paper, 11pt, reqno]{amsart}

\usepackage[utf8]{inputenc}
\usepackage[T1]{fontenc}
\usepackage{mathtools}
\usepackage{lmodern}
\usepackage[english]{babel}
\usepackage{booktabs}
\usepackage{float}
\usepackage{enumitem}
\usepackage{fullpage}
\usepackage{tikz,pgf}
\usetikzlibrary{calc}
\usetikzlibrary{through}
\usepackage[all]{xy}
\usepackage{palatino}
\usepackage{graphicx}
\usepackage{xcolor}
\usepackage{xurl}
\usepackage{hyperref}
\hypersetup{
    colorlinks,
    linkcolor={red!50!black},
    citecolor={blue!50!black},
    urlcolor={blue!80!black}
}
\usepackage[nameinlink]{cleveref}
\usepackage{orcidlink}

\theoremstyle{plain}
\newtheorem{lemma}{Lemma}[section]
\newtheorem{prop}[lemma]{Proposition}
\newtheorem{theorem}[lemma]{Theorem}
\newtheorem{corollary}[lemma]{Corollary}
\theoremstyle{definition}
\newtheorem{definition}[lemma]{Definition}
\newtheorem{es}[lemma]{Example}
\newtheorem{rmk}[lemma]{Remark}
\newtheorem{assumption}[lemma]{Assumption}

\tikzstyle{point}=[circle, draw, fill=black, inner sep=0pt, minimum size=4pt]
\tikzstyle{line}=[line width=1.5pt, black!70!white]

\newcommand{\N}{\mathbb{N}}

\newcommand{\R}{\mathbb{R}}
\newcommand{\C}{\mathbb{C}}
\newcommand{\p}{\mathbb{P}}

\newcommand{\calT}{\mathcal{T}}
\newcommand{\calE}{\mathcal{E}}

\title{Free plane curves with a linear Jacobian syzygy}

\author[Valentina Beorchia]{Valentina Beorchia$^{\circ}$}
\address[\textsc{Valentina Beorchia \orcidlinkc{0000-0003-3681-9045}}]{University of Trieste,
Department of Mathematics, Informatics and Geosciences,
Via Valerio 12/1, 34127 Trieste, Italy}
\email{beorchia@units.it}
\thanks{$^{\circ}$The researcher is a member of ``Gruppo Nazionale per le Strutture Algebriche, Geometriche e le loro Applicazioni'', INdAM. She is partially supported by MUR funds: PRIN project GEOMETRY OF ALGEBRAIC STRUCTURES: MODULI, INVARIANTS, DEFORMATIONS, PI Ugo Bruzzo, Project code: 2022BTA242.}
\author[Matteo Gallet]{Matteo Gallet}
\address[\textsc{Matteo Gallet \orcidlinkc{0000-0003-3601-030X}}]{University of Trieste,
Department of Mathematics, Informatics and Geosciences,
Via Valerio 12/1, 34127 Trieste, Italy}
\email{matteo.gallet@units.it}
\author[Alessandro Logar]{Alessandro Logar}
\address[\textsc{Alessandro Logar \orcidlinkc{0000-0001-7963-5110}}]{University of Trieste,
Department of Mathematics, Informatics and Geosciences,
Via Valerio 12/1, 34127 Trieste, Italy}
\email{logar@units.it}
\date{}
\subjclass[2020]{13D02, 13P99, 14H45, 15A21}
\keywords{Free curves, Jacobian syzygies, Hilbert-Burch matrix, Combinatorial commutative algebra}

\begin{document}

\begin{abstract}
The study of planar free curves is a very active area of research, but a structural study of such a class is missing.
We give a complete classification of the possible generators of the Jacobian syzygy module of a plane free curve under the assumption that one of them is linear.
Specifically, we prove that, up to similarities, there are two possible forms for the Hilbert-Burch matrix.
Our strategy relies on a translation of the problem into the accurate study of the geometry of maximal segments of a suitable triangle with integer points.
Following this description, we are able to determine precisely the equations of free curves and the associated Hilbert-Burch matrices.
\end{abstract}

\maketitle

\section{Introduction}
\label{introduction}

Let $R$ be the ring of polynomials $\C[x,y,z]$ and consider $g \in R$, homogeneous of degree~$n$.
To the polynomial~$g$ we can associate the projective curve $V(g) \subset \p^2$ and the $R$-module $\mathrm{Syz}(J_g)$ of Jacobian relations, where $J_g = \left\langle \partial_{x} g, \partial_{y} g, \partial_{z} g \right\rangle$.
The curve $V(g)$ is called \emph{free} if the module~$\mathrm{Syz}(J_g)$ is free.
The latter condition is equivalent to the projective dimension of~$R/J_g$ being~$2$, hence $\mathrm{Syz}(J_g)$ has two generators.

The study of free curves is a very active area of research, also in connection with Terao's conjecture for line arrangements,
asserting that the freeness of a line arrangement only depends on the combinatorics of the intersection lattice (see \cite[Conjecture 4.138]{OrlikTerao1992}).

Although many examples of free curves have been exhibited, see for instance \cite{BuchweitzConca2013,Dimca2018,DiGennaroIlardiEtAl2024,DimcaSticlaru2012,DimcaSticlaru2018a,DimcaSticlaru2018b,DimcaSernesi2014,BeorchiaMiroRoig2024,DimcaIlardiEtAl2023},
a structural study of such a class of curves is missing.

In the present paper, we shall give a complete classification of the possible generators of~$\mathrm{Syz}(J_g)$
under the assumption that one of them is given by a triple of linear forms.
Equivalently, we will explore the locus of Hilbert-Burch matrices associated with free curves, having the property that their three $2 \times 2$ minors yield a vector proportional to the gradient of a homogeneous form.
Thanks to the exactness of a suitable de Rham sequence, this is equivalent to the curl of the three minors being zero.
We translate this condition into a series of bihomogeneous equations in the coefficients of the polynomials appearing in the Hilbert-Burch matrix, which we study accurately with methods of combinatorial commutative algebra.

Our main result is the following.
Up to similarities, there are two possible forms for the Hilbert-Burch matrix.
The first one is the most interesting for us, and corresponds to the case where the matrix is of the form
\[
    \begin{pmatrix}
        ax & D(y,z) \\
        by & E(x,y,z) \\
        cz & F(x,y,z)
    \end{pmatrix}
\]
where $a\not =0, b, c \in \C$, and $D$, $E$, $F$ are homogeneous forms of degree $n-2$.
The core of the paper is a careful analysis of the curl equations.
Useful tools are the set $\calT \subset \N^3$ of integer points in the convex hull of $(n,0,0)$, $(0,n,0)$, and~$(0,0,n)$, and $\calT'$, namely $\calT$ with the three vertices removed.
Our work allows one to explicitly express the solutions of the curl equations in terms of maximal segments of~$\calT'$.
In this way, we obtain \Cref{prop:explicit_formulas} which precisely determines the polynomials $D$, $E$, and~$F$.
We then select those solutions corresponding to reduced curves that are not unions of concurrent lines (\Cref{lemma:reduced_segments} and \Cref{theorem:characterization_free_curves}).
Finally, we gather our findings in \Cref{table:cases}, which collects (up to permutations of the variables) all the possible free curves with a linear Jacobian syzygy in the considered case.\\
The second possible form for the Hilbert-Burch matrix is
\[
    \begin{pmatrix}
        y & D(x,z) \\
        z & E(x,y,z) \\
        0 & F(x,y,z)
    \end{pmatrix}
\]
where $D$, $E$, $F$ are homogeneous forms of degree $n-2$.
This case yields a single class of free curves, i.e., unions of hyperosculating conics,
possibly together with their common tangent, the so-called \emph{P{\l}oski curves}.
This class has been thoroughly studied, see \cite{Ploski2014,Shin2016,Cheltsov2017}.

Free curves with a linear Jacobian syzygy belong to the wider class of curves with a positive dimensional automorphism group, and they have been classified in
\cite{AluffiFaber2000} and \cite[Proposition~4.4]{PatelRiedlTseng2024}.
Our results specify precisely which curves are the free ones, and allow one to describe the corresponding Hilbert-Burch matrices; see \Cref{section:nearly_free} for more details.

Finally, as an application, we give a characterization of free curves with a linear Jacobian syzygy and only quasi-homogeneous singularities, see \Cref{theorem:ploski}.

The organization of the paper is as follows.
\Cref{notation} sets up the notation and assumptions that we are going to use throughout the paper.
\Cref{linear_syzygy} describes, up to similarities, the two possible cases for the Hilbert-Burch matrices of the syzygy modules of free curves.
\Cref{general_case} introduces the triangle~$\calT$ described above and the relations between its maximal segments and the curl equations.
\Cref{nonconcurrent} provides further restrictions on the solutions since we neglect unions of concurrent lines.
\Cref{squarefree} concludes the study of the general case precisely selecting the polynomials we are interested in.
\Cref{special_case} addresses the second case for the Hilbert-Burch matrix.
\Cref{applications} concludes the paper presenting some applications regarding quasi-homogeneous singularities and nearly free curves.

\section{Notation and preliminaries}
\label{notation}

We set $R := \C[x,y,z]$ and we denote by $C = V(g)$ a reduced plane curve of degree $n$ defined by a homogeneous polynomial $g \in R_n$ in the complex projective plane $\p^2=\mathrm{Proj}(R)$.

We denote by $\mathrm{Syz}(J_g)$ the graded $R$-module of all Jacobian relations for $f$, that is
\[
    \mathrm{Syz}(J_g):=
    \bigl\{
        (a,b,c)\in R^3 \mid
        a \, \partial_x g + b \,\partial_y g + c \, \partial_z g = 0
    \bigr\} \,.
\]
It is well known that the three partials $\partial_x g$, $\partial_y g$ and $\partial_z g$ are linearly dependent if and only if $C$ is
a union of lines passing through a point in~$\p^2$ (i.e., a union of concurrent lines).
So, we make the following assumption.

\begin{assumption}
\label{assumption:no_concurrent_and_squarefree}
From now on, we assume that the three partials $\partial_x g$, $\partial_y g$ and $\partial_z g$ are linearly independent.
Moreover, if $g_{\mathrm{red}}$ generates $\sqrt{(g)}$, then
\[
    \mathrm{Syz}(J_g) \cong \mathrm{Syz}(J_{g_{\rm red}}),
\]
so without loss of generality from now on we will assume that $g$ is square-free.
\end{assumption}

\begin{definition}
\label{def:free_curves}
Let $C = V(g)$ be a reduced singular plane curve of degree $n$.
We say that $C$ is \emph{free} if the graded $R$-module
$\mathrm{Syz}(J_g)$ of all Jacobian relations for $g$ is a free $R$-module.
If
\[
    \mathrm{Syz}(J_g) = R(-n_1)\oplus R(-n_2)
\]
with $n_1+n_2=n-1$, the integers $(n_1,n_2)$ are called the \emph{exponents} of $C$.
In this case, the Jacobian ideal admits a minimal free resolution of the type
\[
    \xymatrix{
        0 \ar[r] & R(-n_1) \oplus R(-n_2) \ar[r] & 3R\bigl(-(n-1)\bigr) \ar[r] & J_g \ar[r] & 0
    }
\]
\end{definition}

\begin{rmk}
The freeness condition for a reduced singular curve $C=V(g)$ in $\p^2$ that is not a set of concurrent lines,
is equivalent to the Jacobian ideal~$J_g$ being arithmetically Cohen-Macaulay of codimension two;
such ideals are completely described by the Hilbert-Burch Theorem (see, for instance, \cite{Eisenbud1995}):
if $I = \left\langle g_1,\cdots, g_m \right\rangle \subset R$ is a Cohen-Macaulay ideal of codimension two, 
then $I$ is defined by the maximal minors of the $(m-1)\times m$ matrix of the first syzygies of the ideal $I$.
Combining this with Euler's formula for a homogeneous polynomial,
we get that a free curve $C=V(g)$ in $\p^2$ has a very constrained structure:
there exists a $3 \times 3$ matrix~$M$, with one column consisting of the $3$ variables,
and the remaining $2$ rows a minimal set of generators of $\mathrm{Syz}(J_g)$, such that $\det (M)\equiv 0 \mod (g)$.
\end{rmk}

We now explore the Hilbert-Burch matrices of free curves.
We start by characterizing triples of polynomials that are the gradient of a homogeneous form (see \cite[Lemma~4.2]{BeorchiaGaluppiVenturello2021}):

\begin{lemma}
\label{lemma:de_rham}
A triple of polynomials $(G_1,G_2,G_3)$ satisfies the relations
\[
    \partial_x G_2 - \partial_y G_1 = 0, \quad
    \partial_x G_3 - \partial _z G_1 = 0, \quad
    \partial _y G_3 - \partial_z G_2 = 0
\]
if and only if there exists an element $H \in \Omega^0_{\mathbb{A}^{3}}\cong R$ such that
\[
    \nabla H = (G_1,G_2,G_3).
\]
If, moreover, the polynomials $G_1,G_2,G_3$ are homogeneous of degree $n-1$, then by Euler's formula the polynomial~$H$ is homogeneous of degree~$n$.
\end{lemma}

As a consequence, given a generic matrix $M = \begin{pmatrix} A & D\\
B & E\\
C & F\\
\end{pmatrix}$, if we set
\[
    G_1 := BF-CE, \quad
    G_2 := CD-AF, \quad
    G_3 := AE-BD,
\]
we have that $M$ is a first syzygy matrix for some free plane curve $V(g)$ if and only if
the triple $(G_1: G_2: G_3 )$ is proportional to $(\partial_x g: \partial_y g: \partial_z g)$.
Moreover, by \Cref{lemma:de_rham}, this holds if and only if the curl $(K_1:K_2:K_3)$ of $(G_1:G_2:G_3 )$ is identically zero, where:
\begin{equation}
\label{eq:curl}
    K_1 := \partial_x G_2 - \partial_y G_1, \quad
    K_2 := \partial_x G_3 - \partial _z G_1 , \quad
    K_3 := \partial _y G_3 - \partial_z G_2.
\end{equation}
Such expressions are homogeneous polynomials of degree $n-2$ and they give the zero polynomials if and only if
the coefficients of all monomials in $x,y,z$ are zero.
These coefficients are, in turn, bihomogeneous polynomials of bidegree $(1,1)$ in the coefficients of the entries of~$M$.

\section{Free curves with a linear syzygy}
\label{linear_syzygy}

We focus now on the locus of Hilbert-Burch matrices of type
\begin{equation}
\label{eq:hilbert_burch_matrix}
    M =
    \begin{pmatrix}
        A & D\\
        B & E\\
        C & F\\
    \end{pmatrix},
\end{equation}
where $A,B,C \in \C[x,y,z]_1$, i.e., are linear forms.

\begin{rmk}
\label{rmk: no two zero entries}
Since, by \Cref{assumption:no_concurrent_and_squarefree}, we are not considering unions of concurrent lines,
we can assume that at most one of the linear forms $A,B,C$ is zero, and similarly at most one of the forms $D,E,F$ is zero.
Indeed, otherwise one minor of the matrix \eqref{eq:hilbert_burch_matrix} would be identically zero.
\end{rmk}

Projective plane transformations induce transformations on linear Jacobian syzygies as described in the following result.

\begin{lemma}
Let $A, B, C \in \C[x,y,z]_1$ and let $H \in \C^{3 \times 3}$ be such that
$H \cdot
\left(
  \begin{smallmatrix}
    x\\ y\\ z\\
  \end{smallmatrix}
\right) =
\left(
  \begin{smallmatrix}
    A\\ B\\ C\\
  \end{smallmatrix}
\right)$.
A polynomial $g\in \C[x,y,z]_n$ admits a linear Jacobian syzygy of type
\[
    A \partial_x g + B \partial_y g + C \partial_z g = 0
\]
if and only if for any $N\in GL(3,\C)$ the triple $N^{-1} \cdot H \cdot N \cdot {}^t (x \ y \ z)$
gives a linear Jacobian syzygy of the polynomial
$g(N\cdot {}^t (x \ y \ z))$.
\end{lemma}

\begin{proof}
By the chain rule one gets
\begin{equation}
\label{eq:nabla}
    (\nabla g) \bigl(N \cdot {}^t (x \ y \ z)\bigr) =
    \nabla \Bigl(g\bigl(N\cdot {}^t (x \ y \ z)\bigr)\Bigr) \cdot N^{-1}.
\end{equation}
Moreover, we have
\[
    \begin{pmatrix}
        A \bigl( N\cdot {}^t (x \ y \ z) \bigr)\\
        B \bigl( N\cdot {}^t (x \ y \ z) \bigr)\\
        C \bigl( N\cdot {}^t (x \ y \ z) \bigr)
    \end{pmatrix}
    = H \cdot N \cdot {}^t (x \ y \ z),
\]
and the relation $\nabla g \cdot
\begin{pmatrix}
  A\\ B\\ C\\
\end{pmatrix} = 0$ gives
\[
    (\nabla g) \bigl( N\cdot {}^t (x \ y \ z) \bigr)\cdot
    \begin{pmatrix}
        A \bigl( N\cdot {}^t (x \ y \ z) \bigr)\\
        B \bigl( N\cdot {}^t (x \ y \ z) \bigr)\\
        C \bigl( N\cdot {}^t (x \ y \ z) \bigr)
    \end{pmatrix} = 0.
\]
By taking into account \eqref{eq:nabla}, we finally get
\[
    \nabla \Bigl(g\bigl(N\cdot {}^t (x \ y \ z)\bigr)\Bigr) \cdot N^{-1} \cdot H \cdot N \cdot {}^t (x \ y \ z) = 0. \qedhere
\]
\end{proof}

As a consequence, up to a linear change of coordinates, we can assume that $N$ is in Jordan block form.
Thus, any free curve with a linear Jacobian syzygy is projectively equivalent to one with linear Jacobian syzygy of one of the following three forms:
\begin{align}
    \label{eq:3_blocks}
    A &= ax, & B &= by, & C &= cz, \\
    \label{eq:2_blocks}
    A &= ax+y, & B &= ay, & C &= cz, \\
    \label{eq:1_block}
    A &= ax+y, & B &= ay+z, & C &= az.
\end{align}

\begin{rmk}
Observe that in case~\eqref{eq:3_blocks} we can assume that the three coefficients $a,b,c$ are not all equal, otherwise we would get the Euler relation.
\end{rmk}

\begin{lemma}
\label{lemma:second_and_third_case}
\mbox{}
\begin{enumerate}
    \item
    If a polynomial $g$ admits the linear Jacobian syzygy $(ax+y, ay, cz)$, i.e., we are in case~\eqref{eq:2_blocks},
    then either $g$ is a monomial, or its zero locus is a set of concurrent lines, or it is the zero polynomial.
    \item
    If a polynomial $g$ admits the linear Jacobian syzygy $(ax+y, ay+z, az)$, i.e., we are in case~\eqref{eq:1_block}, and $a \neq 0$, then $g = 0$.
\end{enumerate}
\end{lemma}

\begin{proof}
We express a generic degree~$n$ polynomial~$g$ as follows
\begin{equation}
\label{eq:generic_g}
    g = \sum_{(i,j,k) \in I} g_{ijk}x^iy^jz^k,
    \quad
    I = \bigl\{ (i, j, k) \mid 0 \leq i, j, k \leq n, \,i+j+k = n \bigr\}.
\end{equation}
\textbf{Case (1)}.
From the relation $(ax+y)\,\partial_x g +ay\,\partial_y g +cz\,\partial_z g =0$ and the Euler identity, we get
\[
    a \,n \,g + y\,\partial_x g +(c-a)z \,\partial_z g = 0.
\]
By \eqref{eq:generic_g}, we get
\[
    a\,n \sum_{(i,j,k) \in I} g_{ijk}x^iy^jz^k+y\,
    \sum_{(i,j,k) \in I, \,i\ge 1} i\,g_{ijk}x^{i-1}y^jz^k +
    (c-a)z \, \sum_{(i,j,k) \in I, \,k\ge 1} k \,g_{ijk}x^iy^j z^{k-1}= 0,
\]
that is
\[
    a\,n \sum_{(i,j,k) \in I} g_{ijk}x^iy^jz^k +
    \sum_{(i,j,k) \in I, \,i\ge 1} i\,g_{ijk}x^{i-1}y^{j+1}z^k +
    \sum_{(i,j,k) \in I} (c-a)k\,g_{ijk}x^iy^{j} z^{k} = 0.
\]
Let $v$ be the left hand side of the previous equality.
Fix $k \leq n$ and consider the monomials:
\[
  x^{n-k}z^k, \quad x^{n-k-1}yz^k, \quad x^{n-k-2}y^2z^k,\ \dots\ , \ y^{n-k}z^k
\]
(i.e., all the degree $n$ monomials with $z$ of fixed degree $k$).
Let $\alpha_k := a(n-k)+ck$.
The coefficients in $v$ of the above monomials are, respectively:
\begin{align*}
    \alpha_k g_{n-k\,0\,k}, \\
    \alpha_k g_{n-k-1\,1\,k}+(n-k)g_{n-k\, 0\, k}, \\
    \alpha_k g_{n-k-2\,2\,k}+(n-k-1)g_{n-k-1\, 1\, k},\\
    \vdots \\
    \alpha_k g_{0\,n-k\,k}+g_{1\, n-k-1\, k}.
\end{align*}
Since $v=0$, all these linear forms in $g_{ijk}$ must be zero.
If $\alpha_k \neq 0$, then we see that $g_{ijk} = 0$ for all $i, j$ (where $i+j +k = n$).
If, conversely, $\alpha_k = 0$, then we get that $g_{ijk} = 0$ for $i, j$ such that $i + j + k = n$ and $i > 0$,
while $g_{0\, n-k\, k}$ can take any value.
Therefore, we distinguish three subcases:
\begin{description}
    \item[Subcase A]
    If $a(n-k)+ck \neq 0$ for all $k$, then $g$ is the zero polynomial.
    \item[Subcase B]
    If for a single value~$\bar{k}$ of $k$ we have $a(n - \bar{k})+\bar{k}c = 0$, while for all $k \neq \bar{k}$ we have $a(n-k)+ck \neq 0$, then
    $g$ is the monomial $g_{0\, n - \bar{k}\, \bar{k}}y^{n -\bar{k}}z^{\bar{k}}$.
    \item[Subcase C]
    If $a(n-k)+kc = 0$ is zero for at least two different values $\bar{k}_1$ and~$\bar{k}_2$, then $a = c = 0$.
    Thus, we get the Jacobian relation $y\partial_x g=0$, which implies $\partial_x g=0$ and $V(g)$ is a set of concurrent lines.
\end{description}

\textbf{Case (2)}.
We argue in a similar way.
From the relation $(ax+y)\,\partial_x g + (ay+z)\,\partial_y g + az\,\partial_z g =0$ and the Euler identity, we get
\[
    a \,n \,g + y\,\partial_x g +z \,\partial_y g = 0.
\]
As in Case (1), by \eqref{eq:generic_g}
we get the following square homogeneous linear system in the variables~$\{g_{ijk}\}$:
\[
    \left\{
        \begin{aligned}
            a\,n \,g_{n00} &= 0\\
            a\,n \,g_{0n0}+g_{1 n-1 0} &= 0\\
            a\,n \,g_{0jk}+g_{1jk}+ j\,g_{0\, j+1 \,k-1}&=0, &k &\ge 1\\
            a\,n \,g_{i0k}+ i\,g_{i\,1 \,k-1}&=0, & k &\ge 1\\
            a\,n \,g_{ij0}+ (i+1)\,g_{i+1\,j-1 \,0}&=0, \\
            a\,n \,g_{ijk}+ i\,g_{i+1\,j-1 \,k}+ j\,g_{i\,j+1 \,k-1}&=0, & j &\ge 1, \,k\ge 1.
        \end{aligned}
    \right.
\]
Order the indeterminates~$g_{ijk}$ lexicographically, i.e.,
$g_{i_1, j_1, k_1} \leq g_{i_2, j_2, k_2}$ if and only if
$i_1 < i_2$ or ($i_1 = i_2$ and $j_1 < j_2$) or ($i_1 = i_2$, $j_1 = j_2$, and $k_1 < k_2$).
In this way, the associated matrix is upper triangular, with $a\,n$ on the diagonal.
Hence, if $a \neq 0$, the matrix is of maximal rank,
therefore, the system admits only the trivial solution.
\end{proof}

We sum up the findings of this Section in the following proposition.

\begin{prop}
Let $g \in \C[x,y,z]_n$ be a non-zero polynomial defining a reduced curve that is not a set of concurrent lines, with a linear Jacobian syzygy $(A, B, C)$.
Up to linear changes of variables, we have the following two possibilities:
\begin{enumerate}
    \item $A = ax$, $B = by$, $C = cz$ for some $a, b, c$ not all equal and $a \neq 0$;
    \item $A = y$, $B = z$, $C = 0$.
\end{enumerate}
\end{prop}

\Cref{general_case,nonconcurrent,squarefree}
concern the study of the first of the two cases above, while
\Cref{special_case} briefly analyzes the second case.

\section{The case \texorpdfstring{$A = ax$, $B = by$, $C = cz$}{A=ax, B=by, C=cz}}
\label{general_case}

This Section is devoted to an accurate analysis of the general case.
The final point is to obtain the possible expressions of squarefree polynomials $g$ given by:
\begin{equation}
\label{eq:def_g}
    g = \frac{1}{n}\det
    \begin{pmatrix}
        x & ax & D \\
        y & by & E \\
        z & cz & F
    \end{pmatrix}
\end{equation}
and annihilating the curl from \Cref{eq:curl}.
The final result is summarized in~\Cref{squarefree},~\Cref{table:cases}.

To clarify our strategy, we start by illustrating the main ideas and the results of the next three Sections.

Observe that the expressions in \eqref{eq:curl} are bihomogeneous in the variables $a,b,c$ and in the coefficients of $D,E,F$; see \Cref{prop:6eqs} for the explicit form.
With suitable elementary operations, the system of equations can be transformed in an equivalent, simpler one, composed of four families:
three of them split into a product of a linear form in~$a$, $b$, $c$ and a coefficient of~$D$, or~$E$, or~$F$,
while the fourth one is given by linear polynomials in the coefficients of $E$ and~$F$; see \Cref{prop:4eqs}.

If the linear forms in $a,b,c$ never vanish, the system admits only the trivial solution.
Therefore, we focus on the other situation.

The coefficients of the linear forms determine lattice points in a triangle in~$\R^3$, see \Cref{def:grid} and \Cref{fig:grid}.
The existence and dimension of the space of solutions can be translated to the geometry of segments in the lattice, see \Cref{solving} and \Cref{fig:primary_decomposition}.
Finally, we show in \Cref{cor:primary_decomposition} that particular lattice points and suitable maximal segments correspond to free polynomials~$g$.

Moreover, we shall prove that it is possible to explicitly express the general~$g$ satisfying the linear Jacobian syzygy,
and to exclude the situations of concurrent lines (see \Cref{nonconcurrent}) and of nonreduced polynomials, see \Cref{squarefree}.

\subsection{The equations}

The first task is to codify freeness of a form with exponents $(1, n-2)$ into a system of polynomial equations.
Consider the matrix $M$ given in \eqref{eq:hilbert_burch_matrix} with $A=ax$, $B=by$, $C=cz$.
As $a,b,c$ are not all zero, we will assume from now on that
\[
    a \neq 0.
\]
Moreover, we observe that,
by adding to the second column of $M$ a suitable multiple of the first column, we
can assume that $D$ is a generic form of degree $n-2$ in $y, z$ without changing the $2 \times 2$ minors of~$M$.
Let $E$, $F$ be generic forms in $x, y, z$ of degree~$n-2$.
Hence, the forms~$E$, $F$, and~$D$ can be written as:
\begin{align*}
    E &= \sum_{(i,j,k) \in I} e_{ijk}x^iy^jz^k,
    & I &= \{(i, j, k) \mid 0 \leq i, j, k \leq n-2,
    \ i+j+k = n-2\}, \\
    F &= \sum_{(i,j, k) \in I} f_{ijk}x^iy^jz^k,
    & I &= \{(i, j, k) \mid 0 \leq i,j, k \leq n-2,
    \ i+j+k = n-2\}, \\
    D &= \sum_{(j,k) \in I} d_{0jk}y^jz^k,
    & I &= \{(j, k) \mid 0 \leq j, k \leq n-2,
    \ j+k = n-2 \}.
\end{align*}
We consider now the equations $K_1=0, K_2=0, K_3=0$, where the $K_i$ are as in \Cref{eq:curl}.
We introduce a geometric object that helps us to keep track of the exponents of monomials appearing in such equations, to simplify their handling.

\begin{definition}
\label{def:grid}
Let $n \in \N$.
The \emph{$n$-triangle grid} $\calT_n$ is the set of integer points of the $n$-th dilation of the standard $2$-simplex in $\R^3$,
i.e., the integer points of the convex hull of the three points $(n,0,0)$, $(0,n,0)$, and~$(0,0,n)$.
A \emph{maximal segment} in $\calT_n$ is the intersection of $\calT_n$ with a line connecting two points in the grid different from $(n,0,0)$, $(0,n,0)$, and~$(0,0,n)$;
these sets are indeed maximal with respect to inclusion.
\end{definition}

\begin{prop}
\label{prop:6eqs}
The polynomial $K_1$ is identically zero if and only if
for every $i, j, k \geq 0$ such that $i+j+k = n-2$
the following equations are satisfied:
\begin{align}
    \bigl((i+1)a+(j+1)b\bigr)f_{ijk} &= 0 \quad (\text{if } k = 0),
    \label{eq:eqn1}\\
    \bigl((i+1)a+(j+1)b \bigr)f_{ijk}-(j+1)c \, e_{i j+1\, k-1} &= 0 \quad
    (\text{if } k > 0).
    \label{eq:eqn2}
\end{align}
The polynomial $K_2$ is identically zero if and only if
for every $i, j, k \geq 0$ such that $i+j+k = n-2$
the following equations are satisfied:
\begin{align}
    \bigl((i+1)a+(k+1)c\bigr)e_{ijk} &= 0 \quad (\text{if } j = 0),
    \label{eq:eqn3}\\
    \bigl((i+1)a+(k+1)c \bigr)e_{ijk}-(k+1)bf_{i j-1\, k+1} &= 0 \quad
    (\text{if } j > 0).
    \label{eq:eqn4}
\end{align}
The polynomial $K_3$ is identically zero if and only if
for every $i, j, k \geq 0$ such that $i+j+k = n-2$
the following equations are satisfied:
\begin{align}
    \bigl((j+1)b+(k+1)c\bigr)d_{ijk} &= 0 \quad (\text{if } i = 0),
    \label{eq:eqn5}\\
    (j+1)e_{i-1\,j+1\,k}+(k+1)f_{i-1\, j\, k+1} &= 0 \quad
    (\text{if } i > 0).
    \label{eq:eqn6}
\end{align}
\end{prop}

\begin{proof}
A direct computation gives
\[
    K_1 = -(a+b)F
    + c z \partial_y E
    - a x\partial_x F
    - b y \partial_y F \,.
\]
The monomials of the polynomial $\partial_x F$ (a form of degree $n-3$) are represented by the set~$\calT_{n-3}$;
specifically, to the point $(i, j, k) \in \calT_{n-3}$ we associate the coefficient~$(i+1) f_{i+1\,jk}$.
Hence, the monomials of~$x\, \partial_x F$ are represented by the points of~$\calT_{n-2}$;
more precisely, to the point $(i, j, k) \in \calT_{n-2}$ (which comes,
if $i>0$, from the point $(i-1, j, k) \in \calT_{n-3}$) it is associated the coefficient~$i f_{ijk}$.
Note that this formula holds also for the case $i = 0$.
In a similar way, to the point $(i, j, k) \in \calT_{n-2}$ we associate the
coefficient~$j f_{ijk}$ of the polynomial~$y\, \partial_y F$.
For the polynomial~$z\, \partial_y E $, we distinguish two cases:
for the points $(i, j, k) \in \calT_{n-2}$ with $k=0$, the coefficient of $z\, \partial_y E$ is $0$,
while, if $k > 0$, the coefficient is $(j+1)e_{ij+1\, k-1}$.
The polynomial $K_1$ is a form of degree~$n-2$, hence
its monomials are encoded in $\calT_{n-2}$.
Let $(i, j, k) \in \calT_{n-2}$ with $k=0$.
Then, the coefficient in~$K_1$ of the
monomial $x^iy^jz^k$ (which is $x^iy^j$ with $i+j = n-2$) is
\[
    -(a+b)f_{ijk}-aif_{ijk}-bjf_{ijk} =
    -\bigl((i+1)a+(j+1)b \bigr)f_{ijk}
\]
and this gives \eqref{eq:eqn1}.
If $k> 0$, the coefficient of the monomial $x^iy^jz^k$ with $i+j+k = n-2$ and $k > 0$ in~$K_1$ is
$-\left((i+1)a + (j+1)b \right)f_{ijk} + (j+1)ce_{i j+1\, k-1}$ and this gives \eqref{eq:eqn2}.
The other cases are proved in a similar way.
\end{proof}

The next proposition provides useful information to solve \Cref{eq:eqn1,eq:eqn2,eq:eqn3,eq:eqn4,eq:eqn5,eq:eqn6}.

\begin{prop}
\label{prop:4eqs}
If $a \neq 0$, then \Cref{eq:eqn1,eq:eqn2,eq:eqn3,eq:eqn4,eq:eqn5,eq:eqn6} are equivalent to the following four sets of equations:
\begin{align}
\label{eqA1_3}
    e_{i-1\,j\,k-1}(ia+jb+kc) &= 0\quad (\mbox{$i>0, \ k>0,\ i+j+k = n$})\\
    f_{i-1\,j-1\,k}(ia+jb+kc) &= 0 \quad (\mbox{$i>0,\ j>0,\ i+j+k = n$})\\
    d_{i\, j-1\,k-1}(ia+jb+kc) &= 0 \quad (\mbox{$i = 0, \ j > 0, \ k > 0,\ i+j+k = n$})
\end{align}
\begin{equation}
\label{eqA4}
    je_{i-1\,j\,k-1}+kf_{i-1\,j-1\,k} = 0 \quad (\text{if } i, j, k > 0, \ i+j+k = n)
\end{equation}
\end{prop}

\begin{proof}
Let us suppose that \Cref{eq:eqn1,eq:eqn2,eq:eqn3,eq:eqn4,eq:eqn5,eq:eqn6} hold true.
To simplify the notations, we denote by $\calE_1(i, j, k),
\dotsc, \calE_6(i, j, k)$ the left hand side of \Cref{eq:eqn1,eq:eqn2,eq:eqn3,eq:eqn4,eq:eqn5,eq:eqn6} (it is understood
that $i, j, k$ are subject to the required restrictions given in \Cref{prop:6eqs}). 
The equations of the third line of~\eqref{eqA1_3}
are already obtained
in~\eqref{eq:eqn5}, while \Cref{eqA4} is obtained
from~\eqref{eq:eqn6} substituting $i+1$ in place of~$i$, and $j-1$ in place of~$j$.

If in \Cref{eq:eqn2} we eliminate $f_{ijk}$ using \Cref{eq:eqn6}, we get:
\begin{multline}
\label{eq:formulaX1}
    \bigl((i+1)a+(j+1)b \bigr)\calE_6(i+1, j, k-1) - k\calE_2(i, j, k) = \\
    e_{i\, j+1\, k-1}(j+1)\bigl((i+1)a+(j+1)b+kc \bigr)
\end{multline}
(for $k>0$ and $i+j+k = n-2$).
Similarly, from
\Cref{eq:eqn2,eq:eqn3,eq:eqn4,eq:eqn5,eq:eqn6}
we have:
\begin{multline}
\label{eq:formulaX2}
    \bigl((i+1)a+(k+1)c\bigr)\calE_6(i+1, j-1, k)-j\calE_4(i, j, k) = \\
    f_{i\, j-1\, k+1}(k+1)\bigl((i+1)a+jb+ (k+1)c\bigr)
\end{multline}
(for $j > 0$, $i+j+k = n-2$).
Suppose that \Cref{eq:eqn2,eq:eqn6} are
satisfied.
Then, from~\eqref{eq:formulaX1} we get:
\[
    e_{i\, j+1\,k-1}\big((i+1)a+(j+1)b+kc\bigr) = 0
\]
and, if we write $i-1$ in place of $i$ and $j-1$ in place of
$j$, we get that $e_{i-1\, j, k-1}(ia+jb+kc) = 0$ for $i>0, j>0, k>0$ and
$i+j+k = n$.
But this formula, thanks to~\eqref{eq:eqn3}, also holds for
$j=0$, hence \eqref{eq:eqn2} and \eqref{eq:eqn6} imply that the first
set of \Cref{eqA1_3} are satisfied.
Similarly, from~\eqref{eq:formulaX2} follows the second set of \Cref{eqA1_3}.
Now, suppose that, conversely, \Cref{eqA1_3} are satisfied.
Then, the right hand side of~\eqref{eq:formulaX1} is zero and since also \eqref{eqA4}
holds,~\eqref{eq:eqn2} is satisfied.
With similar computations, we can prove~\eqref{eq:eqn4}.
\end{proof}

\begin{figure}
    \centering
    \begin{tikzpicture}[fvertex/.style={circle,inner sep=0pt,minimum size=3pt,fill=black,draw=black,outer sep=1pt},edge/.style={line width=1.5pt,black!60!white}]
    \begin{scope}[scale=0.28]
    	\node[fvertex, label={[]90:$(4,0,0)$}] (0) at (-20, 0) {};
    	\node[fvertex, label={[]90:$(3,1,0)$},label={[]270:$f_{200}$}] (1) at (-10, 0) {};
    	\node[fvertex, label={[]90:$(2,2,0)$},label={[]270:$f_{110}$}] (2) at (0, 0) {};
    	\node[fvertex, label={[]90:$(1,3,0)$},label={[]270:$f_{020}$}] (3) at (10, 0) {};
    	\node[fvertex, label={[]90:$(0,4,0)$}] (4) at (20, 0) {};
    	\node[fvertex, label={[]90:$(3,0,1)$},label={[]0:$e_{200}$}] (5) at (-15, 5) {};
    	\node[fvertex, label={[]90:$(2,1,1)$},label={[]0:$e_{110}$},label={[]270:$f_{101}$}] (6) at (-5, 5) {};
    	\node[fvertex, label={[]90:$(1,2,1)$},label={[]0:$e_{020}$},label={[]270:$f_{011}$}] (7) at (5, 5) {};
    	\node[fvertex, label={[]90:$(0,3,1)$},label={[]0:$d_{020}$}] (8) at (15, 5) {};
    	\node[fvertex, label={[]90:$(2,0,2)$},label={[]0:$e_{101}$}] (9) at (-10, 10) {};
    	\node[fvertex, label={[]90:$(1,1,2)$},label={[]0:$e_{011}$},label={[]270:$f_{002}$}] (10) at (0, 10) {};
    	\node[fvertex, label={[]90:$(0,2,2)$},label={[]0:$d_{011}$}] (11) at (10, 10) {};
    	\node[fvertex, label={[]90:$(1,0,3)$},label={[]0:$e_{002}$}] (12) at (-5, 15) {};
    	\node[fvertex, label={[]90:$(0,1,3)$},label={[]0:$d_{002}$}] (13) at (5, 15) {};
    	\node[fvertex, label={[]90:$(0,0,4)$}] (14) at (0, 20) {};
    \end{scope}
    \end{tikzpicture}
    \caption{The grid for $n=4$ with the corresponding monomials.}
    \label{fig:grid}
\end{figure}
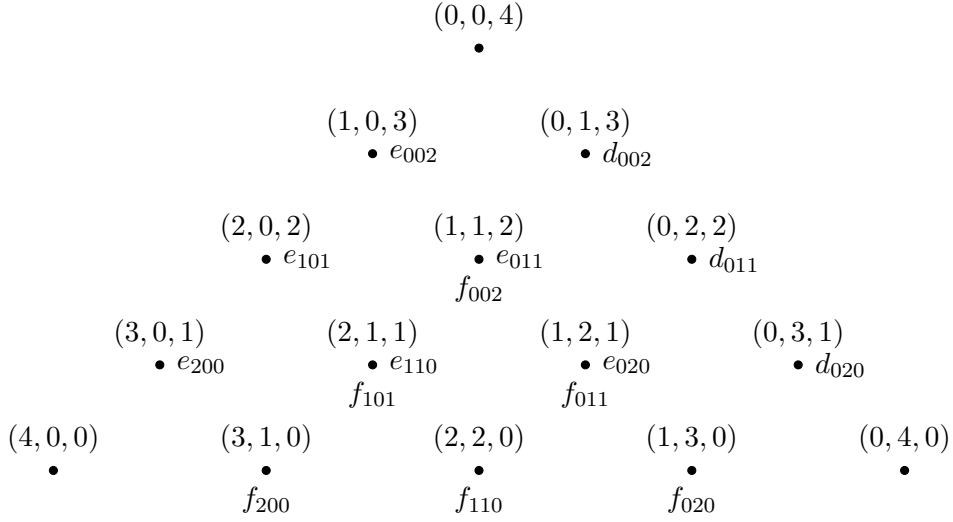

\subsection{Solving \texorpdfstring{\eqref{eqA1_3}}{(15)}, \texorpdfstring{$\dots$}{...}, \texorpdfstring{\eqref{eqA4}}{(18)}}
\label{solving}

Now we determine the explicit solutions of the four sets of equations of \Cref{prop:4eqs}.
To begin with, we determine the solutions of \Cref{eqA1_3}.
One trivial solution is given by $a=b=c=0$, but we discard this case, since we assume $a\not=0$.
Another trivial solution is given by all the coefficients of $D$, $E$ and $F$ equal to zero.
Let $L(i, j, k)$ be the linear form $ia+jb+kc$ in $a, b, c$ and we associate to it the vector $(i, j, k)\in \mathbb{R}^3$.
When $L(i, j, k)$ varies among all the possible cases, the vectors $(i, j, k)$ describe the set~$\calT'$,
which is the $n$-triangle grid $\calT_n$ without the three points $(n, 0, 0)$, $(0, n, 0)$ and $(0, 0, n)$.
Let $\mathcal{U}$ be the set of all the coefficients of $D$, $E$, $F$ together with the element $0$, i.e.:
\[
    \mathcal{U} = \{d_{0jk} \mid (0, j, k) \in \calT_{n-2}\} \cup \{e_{ijk} \mid (i, j, k) \in \calT_{n-2}\}
    \cup \{f_{ijk} \mid (i, j, k) \in \calT_{n-2}\}
    \cup \{0\}
\]
and consider the maps
$\phi_D, \phi_E, \phi_F \colon \calT' \longrightarrow \mathcal{U}$ given by:
\begin{align}
    \phi_D(i, j, k) &=
    \left\{
    \begin{array}{ll}
        d_{0\, j-1\, k-1} & \mbox{if} \ i=0, \ j, k > 0\\
        0 & \mbox{otherwise}
    \end{array}
    \right.
    \label{def:phi_D}\\
    \phi_E(i, j, k) &=
    \left\{
    \begin{array}{ll}
        e_{i-1\, j\, k-1} & \mbox{if} \ i, k > 0\\
        0 & \mbox{otherwise}
    \end{array}
    \right.
    \label{def:phi_E}\\
    \phi_F(i, j, k) &=
    \left\{
    \begin{array}{ll}
        f_{i-1\, j-1\, k} & \mbox{if} \ i, j > 0\\
        0 & \mbox{otherwise}
    \end{array}
    \right.
    \label{def:phi_F}
\end{align}
If $H$ is a subset of $\calT'$, we say that a coefficient $\gamma \in \mathcal{U} \setminus \{0\}$ is
\emph{tied} to $H$ if $\gamma \not \in \phi_D(H)$, $\gamma \not \in \phi_E(H)$, and
$\gamma \not \in \phi_F(H)$.

Similarly as in \Cref{def:grid}, a subset $H$ of $\calT'$ is called a \emph{maximal segment} (of $\calT'$)
if $H$ has at least two points and is a maximal subset of collinear points of $\calT'$.

\begin{figure}
    \centering

    \begin{tikzpicture}[scale = 0.6]
        \coordinate (A) at (0,0);
        \coordinate (A1) at ($(A)+(2,0)$);
        \coordinate (A2) at ($(A)+(4,0)$);
        \coordinate (A3) at ($(A)+(6,0)$);
        \coordinate (B1) at ($(A)+(1,1.732)$);
        \coordinate (B2) at ($(A)+(3,1.732)$);
        \coordinate (B3) at ($(A)+(5,1.732)$);
        \coordinate (B4) at ($(A)+(7,1.732)$);
        \coordinate (C1) at ($(A)+(2,2*1.732)$);
        \coordinate (C2) at ($(A)+(4,2*1.732)$);
        \coordinate (C3) at ($(A)+(6,2*1.732)$);
        \coordinate (D1) at ($(A)+(3,3*1.732)$);
        \coordinate (D2) at ($(A)+(5,3*1.732)$);
        \fill (A1) circle (0.15);
        \fill (A2) circle (0.15);
        \fill (A3) circle (0.15);
        \fill (B1) circle (0.15);
        \fill (B2) circle (0.15);
        \fill (B3) circle (0.15);
        \fill (B4) circle (0.15);
        \fill (C1) circle (0.15);
        \fill (C2) circle (0.15);
        \fill (C3) circle (0.15);
        \fill (D1) circle (0.15);
        \fill (D2) circle (0.15);
        \draw[dashed] (A)--($(A)+(8,0)$);
        \draw[dashed] ($(A)+(8,0)$)--($(A)+(4, 4*1.732)$);
        \draw[dashed] (A)--($(A)+(4, 4*1.732)$);
    \end{tikzpicture}
    \quad \quad
    \begin{tikzpicture}[scale = 0.6]
        \coordinate (A) at (0,0);
        \coordinate (A1) at ($(A)+(2,0)$);
        \coordinate (A2) at ($(A)+(4,0)$);
        \coordinate (A3) at ($(A)+(6,0)$);
        \coordinate (B1) at ($(A)+(1,1.732)$);
        \coordinate (B2) at ($(A)+(3,1.732)$);
        \coordinate (B3) at ($(A)+(5,1.732)$);
        \coordinate (B4) at ($(A)+(7,1.732)$);
        \coordinate (C1) at ($(A)+(2,2*1.732)$);
        \coordinate (C2) at ($(A)+(4,2*1.732)$);
        \coordinate (C3) at ($(A)+(6,2*1.732)$);
        \coordinate (D1) at ($(A)+(3,3*1.732)$);
        \coordinate (D2) at ($(A)+(5,3*1.732)$);
        
        \fill (A1) circle (0.15);
        \fill (A2) circle (0.15);
        \fill (A3) circle (0.15);
        \fill (B1) circle (0.15);
        \fill (B2) circle (0.15);
        \fill (B3) circle (0.15);
        \fill (B4) circle (0.15);
        \fill (C1) circle (0.15);
        \fill (C2) circle (0.15);
        \fill (C3) circle (0.15);
        \fill (D1) circle (0.15);
        \fill (D2) circle (0.15);
        \draw[dashed] (A)--($(A)+(8,0)$);
        \draw[dashed] ($(A)+(8,0)$)--($(A)+(4, 4*1.732)$);
        \draw[dashed] (A)--($(A)+(4, 4*1.732)$);
        
        \draw[line width=0.3mm] (A1) -- (A3);
        \draw[line width=0.3mm] (B1) -- (B4);
        \draw[line width=0.3mm] (C1) -- (C3);
        \draw[line width=0.3mm] (D1) -- (D2);
        
        \draw[line width=0.3mm] (B1) -- (D1);
        \draw[line width=0.3mm] (A1) -- (D2);
        \draw[line width=0.3mm] (A2) -- (C3);
        \draw[line width=0.3mm] (A3) -- (B4);
        
        \draw[line width=0.3mm] (A1) -- (B1);
        \draw[line width=0.3mm] (A2) -- (C1);
        \draw[line width=0.3mm] (A3) -- (D1);
        \draw[line width=0.3mm] (B4) -- (D2);
        
        \draw[line width=0.3mm] (A1) -- (C1);
        \draw[line width=0.3mm] (A1) -- (D1);
        \draw[line width=0.3mm] (A1) -- (C3);
        \draw[line width=0.3mm] (A1) -- (B4);
        \draw[line width=0.3mm] (A1) -- (B3);

        \draw[line width=0.3mm] (A2) -- (B1);
        \draw[line width=0.3mm] (A2) -- (D1);
        \draw[line width=0.3mm] (A2) -- (D2);
        \draw[line width=0.3mm] (A2) -- (B4);
        \draw[line width=0.3mm] (A2) -- (C2);
        
        \draw[line width=0.3mm] (A3) -- (B1);
        \draw[line width=0.3mm] (A3) -- (B2);
        \draw[line width=0.3mm] (A3) -- (C1);
        \draw[line width=0.3mm] (A3) -- (D2);
        \draw[line width=0.3mm] (A3) -- (C3);
        
        \draw[line width=0.3mm] (B1) -- (C3);
        \draw[line width=0.3mm] (B1) -- (D2);
        
        \draw[line width=0.3mm] (B2) -- (D1);
        \draw[line width=0.3mm] (B2) -- (C3);
        
        \draw[line width=0.3mm] (B3) -- (C1);
        \draw[line width=0.3mm] (B3) -- (D2);
        
        \draw[line width=0.3mm] (B4) -- (C1);
        \draw[line width=0.3mm] (B4) -- (C3);
        \draw[line width=0.3mm] (B4) -- (D1);
        
        \draw[line width=0.3mm] (B2) -- (C3);
        \draw[line width=0.3mm] (C1) -- (D2);
        \draw[line width=0.3mm] (C3) -- (D1);
        
    \end{tikzpicture}
    \caption{Case $n=4$. The set $\calT'$ and its maximal segments.}
    \label{fig:Tprime_maximal_segments}
\end{figure}
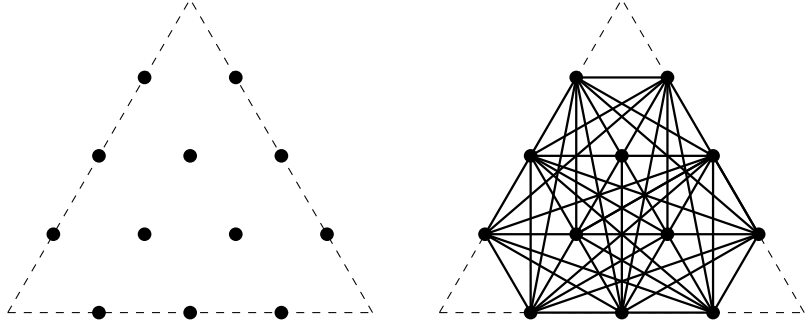

If $H$ is a maximal segment of $\calT'$ or $H$ is given by a single point $(i, j, k) \in \calT'$, then we associate to $H$ the following two sets:
\begin{align}
    m(H) &= \{\gamma \in \mathcal{U}\setminus \{0\} \mid \gamma \mbox{ is tied to } H\},
    \label{def:id_monom}\\
    r(H) &= \{ai+bj+ck \mid  (i, j, k) \in H\}.
\end{align}
The first set consists of some coefficients of the polynomials~$D$, $E$, $F$, while the second one consists of linear forms in $a, b, c$.
\begin{prop}
\label{prop:dec_prim_sol}
The primary decomposition in $\C [a,b,c,\{d_{ijk}\},\{e_{ijk}\},\{f_{ijk}\}]$ of the radical of the ideal
generated by the left hand side of \Cref{eqA1_3}, with $a \neq 0$, is given by the prime ideals
\begin{itemize}
    \item
    $\bigl( \{d_{ijk}\}\cup \{e_{ijk}\}\cup \{f_{ijk}\} \bigr)$;
    \item
    $\bigl( m(H) \cup r(H) \bigr)$
    where $H = \{(i, j, k)\}$ with $(i, j, k) \in \calT'$;
    \item
    $\bigl( m(H) \cup r(H) \bigr)$
    where $H$ is a maximal segment of $\calT'$ whose points are not aligned with~$(n, 0, 0)$.
\end{itemize}
\end{prop}

\begin{proof}
If the system~\eqref{eqA1_3} has a nontrivial solution, then there exist points $(i_1, j_1, k_1), \dotsc, (i_r, j_r, k_r)$ in~$\calT'$ such that
\begin{equation}
\label{eq:lin_sist}
    \left\{
        \begin{aligned}
            a \, i_1 + b \, j_1 + c \, k_1 &= 0\\
            &\vdots \\
            a \, i_r + b \, j_r + c \, k_r &= 0.
        \end{aligned}
    \right.
\end{equation}
This linear system must have a non-trivial solution in $a, b, c$,
so the rank of the associated matrix must be either $1$ or $2$.

If the rank is $1$, then $r=1$, because any two points of $\calT'$ are linearly independent.
In this case, if we set $H = \{ (i_1, j_1, k_1) \}$, the solution to equations \eqref{eqA1_3} is given
by $ai_1+bj_1+ck_1 = 0$ and by
$\gamma = 0$, where $\gamma$ are the coefficients of $D$ or of $E$ or of $F$ tied to $H$;
hence, this solution gives the prime ideal generated by $m(H)\cup r(H)$.

If the rank of the linear system \eqref{eq:lin_sist} is $2$, the points $(i_1, j_1, k_1), \dotsc, (i_r, j_r, k_r)$ are collinear.
Moreover, if $(i_s, j_s, k_s)$ is another point of $\calT'$ collinear with the $r$ previous points, then the linear system~\eqref{eq:lin_sist}
plus the equation $ai_s+bj_s+ck_s = 0$ has the same solutions of~\eqref{eq:lin_sist}.
Therefore, if we have some collinear points in $\calT'$, we can complete them to a maximal set~$H$ of collinear points contained in~$\calT'$.
As soon as the linear equations corresponding to the points of $H$ are satisfied, to find the solutions of the desired equations 
we set to zero all the coefficients of $D$, $E$ and $F$ that are tied to $H$; 
hence, we get again the prime ideal generated by~$m(H) \cup r(H)$.

Finally, if some collinear points of $\calT'$ give~\eqref{eq:lin_sist} of rank~$2$ and these points are collinear with $(n, 0, 0)$,
then $na = 0$, which is not possible.
\end{proof}
\begin{figure}[ht]
    \centering
    \begin{tikzpicture}[scale = 0.7]
        \coordinate (A) at (0,0);
        \coordinate (A1) at ($(A)+(2,0)$);
        \coordinate (A2) at ($(A)+(4,0)$);
        \coordinate (A3) at ($(A)+(6,0)$);
        \coordinate (B1) at ($(A)+(1,1.732)$);
        \coordinate (B2) at ($(A)+(3,1.732)$);
        \coordinate (B3) at ($(A)+(5,1.732)$);
        \coordinate (B4) at ($(A)+(7,1.732)$);
        \coordinate (C1) at ($(A)+(2,2*1.732)$);
        \coordinate (C2) at ($(A)+(4,2*1.732)$);
        \coordinate (C3) at ($(A)+(6,2*1.732)$);
        \coordinate (D1) at ($(A)+(3,3*1.732)$);
        \coordinate (D2) at ($(A)+(5,3*1.732)$);

        \fill (A1) circle (0.15);
        \fill (A2) circle (0.15);
        \fill (A3) circle (0.15);
        \fill (B1) circle (0.15);
        \fill (B2) circle (0.15);
        \fill (B3) circle (0.15);
        \fill (B4) circle (0.15);
        \fill (C1) circle (0.15);
        \fill (C2) circle (0.15);
        \fill (C3) circle (0.15);
        \fill (D1) circle (0.15);
        \fill (D2) circle (0.15);
        
        \draw[line width=0.3mm, dashed] (A1) -- (A3);
        \draw[line width=0.3mm] (B1) -- (B4);
        \draw[line width=0.3mm] (C1) -- (C3);
        \draw[line width=0.3mm] (D1) -- (D2);
        
        \draw[line width=0.3mm, dashed] (B1) -- (D1);
        \draw[line width=0.3mm] (A1) -- (D2);
        \draw[line width=0.3mm] (A2) -- (C3);
        \draw[line width=0.3mm] (A3) -- (B4);
        
        \draw[line width=0.3mm] (A1) -- (B1);
        \draw[line width=0.3mm] (A2) -- (C1);
        \draw[line width=0.3mm] (A3) -- (D1);
        \draw[line width=0.3mm] (B4) -- (D2);
        
        \draw[line width=0.3mm] (A1) -- (C1);
        \draw[line width=0.3mm] (A1) -- (D1);
        \draw[line width=0.3mm] (A1) -- (C3);
        \draw[line width=0.3mm] (A1) -- (B4);
        \draw[line width=0.3mm] (A1) -- (B3);

        \draw[line width=0.3mm] (A2) -- (B1);
        \draw[line width=0.3mm] (A2) -- (D1);
        \draw[line width=0.3mm] (A2) -- (D2);
        \draw[line width=0.3mm] (A2) -- (B4);
        \draw[line width=0.3mm] (A2) -- (C2);
        
        \draw[line width=0.3mm] (A3) -- (B1);
        \draw[line width=0.3mm] (A3) -- (B2);
        \draw[line width=0.3mm] (A3) -- (C1);
        \draw[line width=0.3mm] (A3) -- (D2);
        \draw[line width=0.3mm] (A3) -- (C3);
        
        \draw[line width=0.3mm] (B1) -- (C3);
        \draw[line width=0.3mm] (B1) -- (D2);
        
        \draw[line width=0.3mm] (B2) -- (D1);
        \draw[line width=0.3mm, dashed] (B2) -- (C3);
        
        \draw[line width=0.3mm] (B3) -- (C1);
        \draw[line width=0.3mm] (B3) -- (D2);
        
        \draw[line width=0.3mm] (B4) -- (C1);
        \draw[line width=0.3mm] (B4) -- (C3);
        \draw[line width=0.3mm] (B4) -- (D1);
        
        \draw[line width=0.3mm] (C1) -- (D2);
        \draw[line width=0.3mm] (C3) -- (D1);
        
    \end{tikzpicture}
    
    \caption{The primary decomposition given in \Cref{prop:dec_prim_sol} (for the case $n=4$).
    The black circle are the points of $\calT'$, the lines are the maximal segments of $\calT'$,
    the three dashed lines are three maximal segments that have to be excluded since their
    points are aligned with $(n, 0, 0)$.}
    \label{fig:primary_decomposition}
\end{figure}
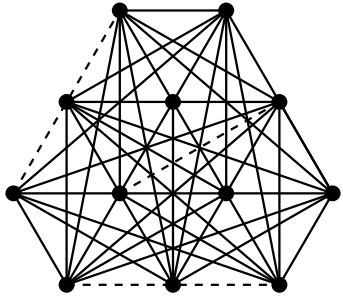

An immediate consequence of the above proposition is the following:

\begin{corollary}
\label{cor:primary_decomposition}
The primary decomposition of the radical of the ideal generated by the left hand side of~\eqref{eqA1_3} and~\eqref{eqA4} is given by the ideals described in~\Cref{prop:dec_prim_sol},
where we make the substitutions $e_{i-1\, j\,k-1} = -(k/j) f_{i-1\, j-1\, k}$ when $i, j, k > 0$.
\end{corollary}

\begin{figure}[ht]
    \centering
    \begin{tikzpicture}[scale = 0.6]
        \coordinate (A) at (0,0);
        
        \fill ($(A)+(2,0)$) circle (0.15);
        \fill ($(A)+(4,0)$) circle (0.15);
        \fill ($(A)+(6,0)$) circle (0.15);
        \fill ($(A)+(1,1.732)$) circle (0.15);
        \fill ($(A)+(3,1.732)$) circle (0.15);
        \fill ($(A)+(5,1.732)$) circle (0.15);
        \fill ($(A)+(7,1.732)$) circle (0.15);
        \fill ($(A)+(2,2*1.732)$) circle (0.15);
        \fill ($(A)+(4,2*1.732)$) circle (0.15);
        \fill ($(A)+(6,2*1.732)$) circle (0.15);
        \fill ($(A)+(3,3*1.732)$) circle (0.15);
        \fill ($(A)+(5,3*1.732)$) circle (0.15);
        \draw[dashed] (A)--($(A)+(8,0)$);
        \draw[dashed] ($(A)+(8,0)$)--($(A)+(4, 4*1.732)$);
        \draw[dashed] (A)--($(A)+(4, 4*1.732)$);
        
        \draw[line width=0.3mm] ($(A)+(1, 1.732)$) -- ($(A)+(7, 1.732)$);
        \draw[line width=0.3mm] ($(A)+(2,0)$) -- ($(A)+(6,2*1.732)$);
        
        \node at (6, 1.3) {$h_1$};
        \node at (5, 3.2) {$h_2$};
    \end{tikzpicture}

    \caption{Two examples of maximal segments.}
    \label{fig:example_max_segments}
\end{figure}
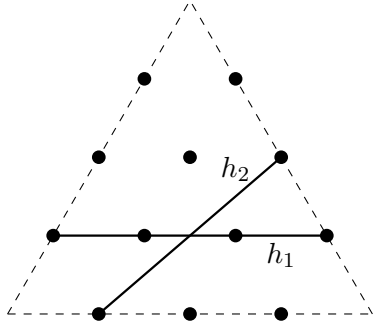

\begin{es}
\label{es:running_example}
Consider the case $n=4$ and the following two maximal segments
(see \Cref{fig:example_max_segments}):
\begin{align*}
    h_1 &= \{(3, 0, 1), (2, 1, 1), (1, 2, 1), (0, 3, 1) \}, \\
    h_2 &= \{(3, 1, 0), (0, 2, 2)\}.
\end{align*}
According to \Cref{def:phi_D,def:phi_E,def:phi_F},
we have:
\[
    \phi_D(h_1) = \{0, d_{020}\}, \quad
    \phi_E(h_1) = \{0, e_{020}, e_{110}, e_{200}\}, \quad
    \phi_F(h_1) = \{0, f_{011}, f_{101}\}
\]
and
\[
    \phi_D(h_2) = \{0, d_{011} \}, \quad
    \phi_E(h_2) = \{0 \}, \quad
    \phi_F(h_2) = \{0, f_{200} \}.
\]
Therefore:
\[
    m(h_1) = \{
        d_{002}, d_{011},
        e_{002}, e_{011}, e_{101},
        f_{002}, f_{020}, f_{110}, f_{200}
    \}
\]
and
\[
    m(h_2) = 
    \{
        d_{002}, d_{020},
        e_{002}, e_{011}, e_{020}, e_{101}, e_{110}, e_{200},
        f_{002}, f_{011}, f_{020}, f_{101}, f_{110}
    \}.
\]
Finally, we can compute the polynomials $D, E, F$ under the conditions $m(h_1)$ and $m(h_2)$:
\[
    D_{h_1} = d_{020}y^2, \quad
    E_{h_1} = e_{020}y^2 + e_{110}xy + e_{200}x^2, \quad
    F_{h_1} = f_{011}yz + f_{101}xz.
\]
\end{es}

It is possible to count the number of prime ideals in the primary decomposition given by~\Cref{prop:dec_prim_sol}.

\begin{lemma}
The number of elements in the primary decomposition given by \Cref{prop:dec_prim_sol} is
\[
  \frac{1}{2}(n-1)(n+4) + \alpha(n+1)-\sum_{i=2}^n\phi(i)-2\sum_{i = \lfloor n/2 \rfloor + 1}^n \phi(i)-1,
\]
where $\phi$ is the Euler totient function.
\end{lemma}

\begin{proof}
The number of elements of $\calT'$ is
\begin{align*}
    p_1 & = \frac{(n-1)(n+4)}{2}.
\end{align*}
The number of lines passing through at least two points of a triangular grid of side $n$ is a known number,
given by
\[
    \alpha(n) = 3 \sum_{j=1}^{n-1} \phi(j)\left(\binom{n-j+1}{2}-\binom{n-2j+1}{2}\right),
\]
see \url{https://oeis.org/A244504} 
(here, we assume $\binom{i}{2}$ is zero if $i \leq 0$).
The points of the lines passing through $N = (n, 0, 0)$ and a point
$Q=(u, v, w) \in \calT'$ (hence with $u+v+w = n$ and $u, v, w \geq 0$) are given by
$N+\lambda (Q-N)$.
If these points have integer coordinates, they are of the form $N + m/d\, Q$,
where $d = \gcd(n-u, v, w)=
\gcd(v+w, v, w) = \gcd(v, w)$ and $m \in \mathbb{Z}$.
From this, we see that the lines of the grid
of side $n+1$ passing through $N$ and $Q$ and no other points of $\calT'$
are in one-to-one correspondence with the points $Q \in \calT'$ such that
$v, w$ are positive, coprime and are such that
$u+v = i$ for $i = 2, 3, \dotsc, n$.
Hence, their number is
\[
    p_2 = \sum_{i=2}^n\phi(i).
\]
These lines have to be subtracted from the computation of the maximal segments of $\calT'$.
Similarly, other lines of the grid of side $n+1$ that do not give maximal segments of $\calT'$
are the lines passing through the point $(0, n, 0)$ and a point $Q = (u, v, w) \in \calT'$ such that
$u, w$ are coprime and $v \geq \lfloor n/2 \rfloor+1$, since these lines contain only one point of
$\calT'$. The number of these lines is given by:
\[
    p_3 = \sum_{i = \lfloor n/2 \rfloor +1}^n \phi(i).
\]
A similar consideration can be given for the lines through $(0, 0, n)$.
Finally, also the two maximal segments
of $\calT'$ given by the line through the two points $(n, 0, 0), (0, n, 0)$ and the maximal segment
given by the line through the points $(n, 0, 0), (0, 0, n)$ must be excluded.
From these computations, we see
that the number of prime ideals of the primary decomposition of~\Cref{prop:dec_prim_sol} is
$1+p_1+\alpha(n+1)-p_2-2p_3-2$, i.e.:
\[
    \frac{1}{2}(n-1)(n+4) + \alpha(n+1) - \sum_{i=2}^n\phi(i) - 2\sum_{i = \lfloor n/2 \rfloor +1}^n \phi(i) - 1. \qedhere
\]
\end{proof}

\section{Restriction to non-concurrent lines solutions}
\label{nonconcurrent}

As observed in \Cref{rmk: no two zero entries}, we want to consider those $H$ given by~\Cref{prop:dec_prim_sol}
(or~\Cref{cor:primary_decomposition}) for which at most one among~$D$, $E$ and~$F$ is zero.
If a point $(i, j, k) \in \calT'$ is such that $i=0$,
then all coefficients of $E$ and all coefficients of $F$ are tied to $H = \{(i, j, k)\}$,
hence in this case $E$ and $F$ evaluate to zero.
Similarly, if $j=0$, then $D$ and $F$ evaluate to zero and if $k=0$, then $D$ and $E$ evaluate to zero.
Therefore, with our hypothesis, if $H$ is a singleton, it must be an interior point of $\calT'$.
For the same reason, the three sets of collinear points of $\calT'$ given by:
\begin{align*}
    H_1 &= \{(0, j, k) \mid j+k = n,\ j, k > 0\} \\
    H_2 &= \{(i, 0, k) \mid i+k = n,\ i, k > 0\} \\
    H_3 &= \{(i, j, 0) \mid i+j = n,\ i, j > 0\}
\end{align*}
are not acceptable maximal segments of $\calT'$ (the last two have already
been excluded, since collinear with the point $(n, 0, 0)$).
Let therefore
\[
    \mathcal{H}_1 :=
    \left\{
        \{(i, j, k)\}
        \mid
        (i,j,k) \in \calT' \text{ and } i, j, k > 0
    \right\}
\]
and
\[
    \mathcal{H}_2 :=
    \left\{
        H \subseteq \calT' \mid
        H \mbox{ is a maximal segment}
    \right\}
    \setminus \{H_1, H_2, H_3\} \,.
\]
We are interested in the values of $D, E, F$ under the conditions
$m(H)$ given by~\eqref{def:id_monom} with $H \in \mathcal{H}_1 \cup \mathcal{H}_2$.
Hence, we have to compute the reduced form of $D, E, F$ w.r.t.\ the monomial
ideal~$\bigl(m(H)\bigr)$.
We denote these values by $D_H, E_H, F_H$.
The discussion above immediately implies the following result.

\begin{prop}
\label{prop:explicit_formulas}
If $H \in \mathcal{H}_1 \cup \mathcal{H}_2$, we have:
\begin{equation}
\label{eq:DEF_reduced}
    \begin{aligned}
        D_H &= \sum_{(i, j, k) \in H}\phi_D(i, j, k) x^iy^{j-1}z^{k-1},\\
        E_H &= \sum_{(i, j, k) \in H}\phi_E(i, j, k) x^{i-1}y^jz^{k-1}\\
        F_H &= \sum_{(i, j, k) \in H}\phi_F(i, j, k) x^{i-1}y^{j-1}z^k
    \end{aligned}
\end{equation}
\end{prop}

\begin{rmk}
In particular, if $H = \{(i, j, k)\} \in \mathcal{H}_1$, then
\begin{equation}
\label{eq:DEF_H_point}
    \begin{aligned}
        D_H &= 0,\\
        E_H &= e_{i-1\, j\, k-1} x^{i-1}y^jz^{k-1},\\
        F_H &= f_{i-1\, j-1\, k} x^{i-1}y^{j-1}z^k.
    \end{aligned}
\end{equation}
If $H \in \mathcal{H}_2$, then $D_H$ is $0$, unless there exist $j_0, k_0$ such
that $(0, j_0, k_0) \in H$; in this case
$D_H = d_{0\, j_0\, k_0}y^{j_0-1}z^{k_0-1}$.
We see therefore that in case
$H\in \mathcal{H}_1$, $D_H$ is zero, while $E_H$ and $F_H$ are non zero
monomials, while if $H\in \mathcal{H}_2$, $D_H$ is either $0$ or a monomial.
\end{rmk}

If we want the values of the polynomials $D$, $E$, $F$ for a solution
of all the four families of equations of~\Cref{prop:4eqs}, corresponding to a set $H \in \mathcal{H}_1 \cup \mathcal{H}_2$, it is enough to modify the function~$\phi_E$ defined in~\eqref{def:phi_E} as follows, taking into account \eqref{eqA4}:
\begin{equation}
\label{def:phi_Estar}
    \phi_E^*(i, j, k) =
    \left\{
        \begin{array}{ll}
            e_{i-1\, j\, k-1} & \mbox{if} \ i, k > 0, j = 0\\
            -(k/j)\, f_{i-1, j-1, k} & \mbox{if} \ i, j, k > 0\\
            0 & \mbox{otherwise}
        \end{array}
    \right.
\end{equation}
Hence, in this case we define $E_H^*$ by:
\[
    E_H^* = \sum_{(i, j, k) \in H}\phi_E^*(i, j, k)x^{i-1}y^jz^{k-1}.
\]

\begin{es}
In \Cref{es:running_example}, we have:
\[
    E_{h_1}^* = -\frac{1}{2}f_{011}y^2-f_{101}xy+e_{200}x^2
\]
and
\[
    D_{h_2} = d_{011}yz, \quad E_{h_2} = E_{h_2}^* = 0, \quad F_{h_2} = f_{200}x^2.
\]
\end{es}

Let us now describe the final shape of the generic polynomial $g$, not corresponding to concurrent lines, in terms of $a, b, c$ and the
coefficients of~$D$, $E$, and~$F$.

\begin{lemma}
Let $H \in \mathcal{H}_1 \cup \mathcal{H}_2$ and let $g_H$ be
the polynomial determined by \eqref{eq:def_g} using the forms~$D_H$, $E_H$, and~$F_H$ given by \eqref{eq:DEF_reduced}. 
Then, we have
\begin{equation}
\label{eq:ngH}
    n\cdot g_H =
    \sum_{(i, j, k) \in H}
    \bigl(
        (c-b)\phi_D(i, j, k) + (a-c) \phi_E(i,j,k) + (b-a) \phi_F(i,j,k)
    \bigr) x^iy^jz^k
\end{equation}
\end{lemma}

\begin{proof}
By \eqref{eq:def_g}, we have
\[
    g = \frac{1}{n} \bigl( (c-b)yz D + (a-c)xz E + (b-a)xy F \bigr).
\]
By substituting
$D_H, E_H, F_H$ from \Cref{eq:DEF_reduced} in place of $D, E, F$,
we get:
\begin{align*}
    n\cdot g_H &= (c-b)\sum_{(i, j, k)\in H}\phi_D(i,j,k)x^iy^jz^k
    +(a-c)\sum_{(i, j, k)\in H}\phi_E(i,j,k)x^iy^jz^k+ \\
    &\phantom{=\ } (b-a)\sum_{(i,j,k) \in H}\phi_F(i, j, k)x^iy^jz^k \\
    & = \sum_{(i, j, k) \in H}\bigl((c-b)\phi_D(i, j, k)+
    (a-c) \phi_E(i,j,k) + (b-a) \phi_F(i,j,k)\bigr) x^iy^jz^k
\end{align*}
which proves the claim.
\end{proof}

\begin{es}
In the case $n=9$, an example of a maximal segment is
\[
    h = 
    \bigl\{
        (0, 2, 7), (1, 3, 5), (2, 4, 3), (3, 5, 1)
    \bigr\}.
\]
From \eqref{eq:ngH}, we have that
\begin{multline*}
    g_h = \frac {1}{9}y^2z\left(
    \bigl((a-c)e_{250}+(b-a)f_{241}\bigr)x^3y^3+
    \bigl((a-c)e_{142}+(b-a)f_{133}\bigr)x^2y^2z^2+\right.\\
    \left. \bigl((a-c)e_{034}+(b-a)f_{025}\bigr)xyz^4+
    (c-b)d_{016}z^6\right).
\end{multline*}
Note that here $g$ is not reduced.
\end{es}

\begin{rmk}
The polynomial $g_H$ is a homogeneous polynomial in $x, y, z$ of degree $n$ and from the above expression,
we see that the only monomials which appear in $g_H$
are those of the form~$x^iy^jz^k$ where $(i, j, k) \in H$;
therefore, the set $H$ in the triangular grid~$\calT'$ that expresses a solution of~\eqref{eqA1_3}
can be used here with a different meaning: it represents, in the triangular grid~$\calT_n$ of degree $n$ monomials, the monomials appearing in $g_H$.
\end{rmk}

\begin{definition}
We set $g_H^*$ to be the polynomial
obtained by substituting $D_H, E_H^*, F_H$ in \eqref{eq:hilbert_burch_matrix}, and $g_H^{**}$ to be the final
value when also
$a, b, c$ are evaluated according to the conditions
given by the ideal generated by $r(H)$.
\end{definition}

Let us analyze explicitly the two cases.

\subsection{Case \texorpdfstring{$H\in \mathcal{H}_1$}{H in H1}}
\label{caseH1}

Here $H = \{(i, j, k)\}$.
From \eqref{eq:DEF_H_point}, we see that
\[
    g_H = \frac{1}{n}\left((a-c)e_{i-1\, j\, k-1}+(b-a)f_{i-1\,j-1\, k} \right)x^iy^jz^k
\]
while
\[
    g_H^* =
    \left(
        \frac{(b-a)}{n} - \frac{k(a-c)}{j}
    \right) f_{i-1 \, j-1 \, k} x^iy^jz^k
\]
and
\[
    g_H^{**} =\left(\frac{(b-a)}{n}-\frac{(k-i)a}{j}-b\right)f_{i-1\,j-1\, k} x^iy^jz^k.
\]

\subsection{Case \texorpdfstring{$H\in \mathcal{H}_2$}{H in H2}}
\label{caseH2}

Taking into account \eqref{eq:DEF_reduced} and \eqref{def:phi_Estar}, if we denote by $g_H^{\ast}$ the polynomial determined by~\eqref{eq:def_g}, we get
\begin{multline*}
    n \cdot g_H^{\ast} = \sum_{(0,j,k) \in H} (c-b) d_{0 \, j-1 \, k-1} y^j z^k
    + \sum_{(i,0,k) \in H} (a-c) e_{i-1 \, 0 \, k-1} x^i z^k\\
    + \sum_{(i,j,0) \in H} (b-a) f_{i-1 \, j-1 \, 0} x^i y^j
    + \sum_{\substack{(i,j,k) \in H \\ i, j, k>0}} \left((b-a) -\frac{k}{j}(a-c) \right) f_{i-1 \, j-1 \, k} x^i y^j z^k.
\end{multline*}
Notice that the first three sums contribute by at most one monomial in $x,y,z$ each.

The last formula shows that there are no relations among the
coefficients of $g_H^{\star}$, so they can get any possible value.
The greatest common divisor of the monomials
$x^iy^jz^k$ for $(i, j, k)\in H$, is given by
$x^uy^vz^w$ where $u$ is the minimum of the integers $i$ such that
$(i,j,k)\in H$, $v$ is the minimum of the integers $j$ such that
$(i,j,k)\in H$ and $w$ is the minimum of the integers $k$ such that
$(i,j,k)\in H$. The restrictions we have on~$H$ imply that this greatest
common divisor is never $1$, hence
$g_H$ (and $g_H^*$ and $g_H^{**}$) is always a multiple of $x$ or of $y$ or of $z$ (which is trivially true in \Cref{caseH1}).

Moreover, we observe that the explicit expressions for $g_H^{**}$ are quite involved, but since we are interested in squarefree polynomials, we shall determine them in the next Section.

\section{Selecting squarefree solutions}
\label{squarefree}

Since we are interested only in square-free polynomials $g$, we have other restrictions to consider.
If $H = \{(i, j, k)\}$, from~\eqref{eq:DEF_H_point} we see that $g_H$ is never
squarefree (unless $n=3$), so the case in which $H$ is a point
(therefore, the whole set $\mathcal{H}_1$) can be set aside.

Hence, let us focus on the case where $H$ is not a singleton.
From now on, we assume $n \geq 4$.

\begin{definition}
In $\calT'$, we select these subsets:
\begin{align*}
    \Gamma_1 &= \{(n-1, 1, 0), (n-1, 0, 1), (n-2, 1, 1)\},\\
    \Gamma_2 &= \{(0, n-1, 1), (1, n-1, 0), (1, n-2, 1)\},\\
    \Gamma_3 &= \{(0, 1, n-1), (1, 0, n-1), (1, 1, n-2)\},\\
    \Delta_1 &= \{(i, j, k) \in \calT' \mid i = 0, 1 \},\\
    \Delta_2 &= \{(i, j, k) \in \calT' \mid j = 0, 1 \},\\
    \Delta_3 &= \{(i, j, k) \in \calT' \mid k = 0, 1 \},
\end{align*}
and we consider the set
\[
    \mathcal{H}_{\mathrm{red}} =
    \bigl\{
        H \in \mathcal{H}_2 \mid
        \text{ there exists } \ell \in \{1, 2, 3\}
        \text{ such that }
        H \cap \Gamma_{\ell} \neq \emptyset
        \text{ and }
        H \cap \Delta_{\ell} \neq \emptyset
    \bigr\}.
\]
\end{definition}

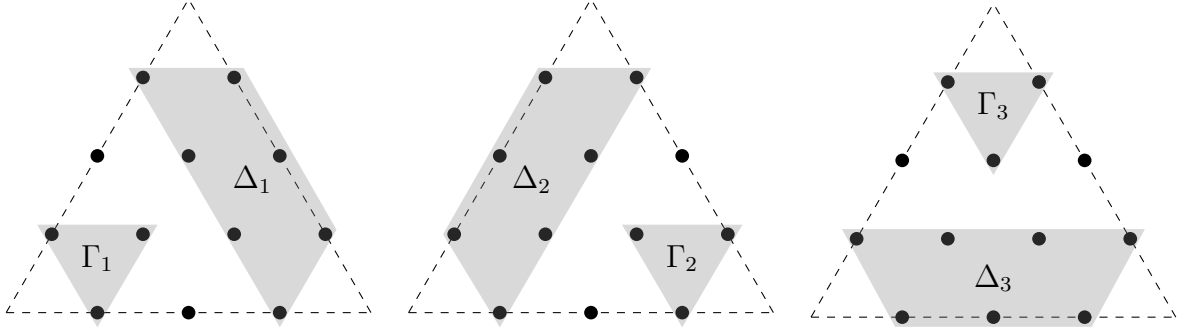
\begin{figure}
    \begin{center}
        \begin{tikzpicture}[scale = 0.6]
            \coordinate (A) at (0,0);
            \coordinate (A1) at ($(A)+(2,0)$);
            \coordinate (A2) at ($(A)+(4,0)$);
            \coordinate (A3) at ($(A)+(6,0)$);
            \coordinate (B1) at ($(A)+(1,1.732)$);
            \coordinate (B2) at ($(A)+(3,1.732)$);
            \coordinate (B3) at ($(A)+(5,1.732)$);
            \coordinate (B4) at ($(A)+(7,1.732)$);
            \coordinate (C1) at ($(A)+(2,2*1.732)$);
            \coordinate (C2) at ($(A)+(4,2*1.732)$);
            \coordinate (C3) at ($(A)+(6,2*1.732)$);
            \coordinate (D1) at ($(A)+(3,3*1.732)$);
            \coordinate (D2) at ($(A)+(5,3*1.732)$);

            \fill (A1) circle (0.15);
            \fill (A2) circle (0.15);
            \fill (A3) circle (0.15);
            \fill (B1) circle (0.15);
            \fill (B2) circle (0.15);
            \fill (B3) circle (0.15);
            \fill (B4) circle (0.15);
            \fill (C1) circle (0.15);
            \fill (C2) circle (0.15);
            \fill (C3) circle (0.15);
            \fill (D1) circle (0.15);
            \fill (D2) circle (0.15);
            \draw[dashed] (A)--($(A)+(8,0)$);
            \draw[dashed] ($(A)+(8,0)$)--($(A)+(4, 4*1.732)$);
            \draw[dashed] (A)--($(A)+(4, 4*1.732)$);
            
            \draw[fill=gray, draw=gray, opacity = 0.3]  ($(B1)+(-0.3,0.2)$) -- ($(B2)+(0.3, 0.2)$) -- ($(A1)+(0, -0.3)$) -- cycle;
            
            \draw[fill=gray, draw=gray, opacity = 0.3]  ($(A3)+(0,-0.3)$) -- ($(B4)+(0.23, 0.1)$) -- ($(D2)+(0.2, 0.2)$) -- ($(D1)+(-0.3,0.2)$) -- cycle;

            \node at (2, 0.7*1.732) {{\large $\Gamma_1$}};
            \node at (5.4, 1.7*1.732) {{\large $\Delta_1$}};
        \end{tikzpicture}
        \quad
        \begin{tikzpicture}[scale = 0.6]
            \coordinate (A) at (0,0);
            \coordinate (A1) at ($(A)+(2,0)$);
            \coordinate (A2) at ($(A)+(4,0)$);
            \coordinate (A3) at ($(A)+(6,0)$);
            \coordinate (B1) at ($(A)+(1,1.732)$);
            \coordinate (B2) at ($(A)+(3,1.732)$);
            \coordinate (B3) at ($(A)+(5,1.732)$);
            \coordinate (B4) at ($(A)+(7,1.732)$);
            \coordinate (C1) at ($(A)+(2,2*1.732)$);
            \coordinate (C2) at ($(A)+(4,2*1.732)$);
            \coordinate (C3) at ($(A)+(6,2*1.732)$);
            \coordinate (D1) at ($(A)+(3,3*1.732)$);
            \coordinate (D2) at ($(A)+(5,3*1.732)$);

            \fill (A1) circle (0.15);
            \fill (A2) circle (0.15);
            \fill (A3) circle (0.15);
            \fill (B1) circle (0.15);
            \fill (B2) circle (0.15);
            \fill (B3) circle (0.15);
            \fill (B4) circle (0.15);
            \fill (C1) circle (0.15);
            \fill (C2) circle (0.15);
            \fill (C3) circle (0.15);
            \fill (D1) circle (0.15);
            \fill (D2) circle (0.15);
            \draw[dashed] (A)--($(A)+(8,0)$);
            \draw[dashed] ($(A)+(8,0)$)--($(A)+(4, 4*1.732)$);
            \draw[dashed] (A)--($(A)+(4, 4*1.732)$);
            
            \draw[fill=gray, draw=gray, opacity = 0.3]  ($(B3)+(-0.3,0.2)$) -- ($(B4)+(0.3, 0.2)$) -- ($(A3)+(0, -0.3)$) -- cycle;
            
            \draw[fill=gray, draw=gray, opacity = 0.3]  ($(A1)+(0,-0.3)$) -- ($(B1)+(-0.23, 0)$) -- ($(D1)+(-0.15, 0.2)$) -- ($(D2)+(0.3,0.2)$) -- cycle;

            \node at (6, 0.7*1.732) {{\large $\Gamma_2$}};
            \node at (2.7, 1.7*1.732) {{\large $\Delta_2$}};
        \end{tikzpicture}
        \quad
        \begin{tikzpicture}[scale = 0.6]
            \coordinate (A) at (0,0);
            \coordinate (A1) at ($(A)+(2,0)$);
            \coordinate (A2) at ($(A)+(4,0)$);
            \coordinate (A3) at ($(A)+(6,0)$);
            \coordinate (B1) at ($(A)+(1,1.732)$);
            \coordinate (B2) at ($(A)+(3,1.732)$);
            \coordinate (B3) at ($(A)+(5,1.732)$);
            \coordinate (B4) at ($(A)+(7,1.732)$);
            \coordinate (C1) at ($(A)+(2,2*1.732)$);
            \coordinate (C2) at ($(A)+(4,2*1.732)$);
            \coordinate (C3) at ($(A)+(6,2*1.732)$);
            \coordinate (D1) at ($(A)+(3,3*1.732)$);
            \coordinate (D2) at ($(A)+(5,3*1.732)$);

            \fill (A1) circle (0.15);
            \fill (A2) circle (0.15);
            \fill (A3) circle (0.15);
            \fill (B1) circle (0.15);
            \fill (B2) circle (0.15);
            \fill (B3) circle (0.15);
            \fill (B4) circle (0.15);
            \fill (C1) circle (0.15);
            \fill (C2) circle (0.15);
            \fill (C3) circle (0.15);
            \fill (D1) circle (0.15);
            \fill (D2) circle (0.15);
            \draw[dashed] (A)--($(A)+(8,0)$);
            \draw[dashed] ($(A)+(8,0)$)--($(A)+(4, 4*1.732)$);
            \draw[dashed] (A)--($(A)+(4, 4*1.732)$);
            
            \draw[fill=gray, draw=gray, opacity = 0.3]  ($(D1)+(-0.3,0.2)$) -- ($(D2)+(0.3, 0.2)$) -- ($(C2)+(0, -0.3)$) -- cycle;
            
            \draw[fill=gray, draw=gray, opacity = 0.3]  ($(B1)+(-0.3,0.2)$) -- ($(B4)+(0.3, 0.2)$) -- ($(A3)+(0.15, -0.2)$) -- ($(A1)+(-0.15,-0.2)$) -- cycle;
               
            \node at (4, 2.7*1.732) {{\large $\Gamma_3$}};
            \node at (4, 0.5*1.732) {{\large $\Delta_3$}};
        \end{tikzpicture}
  
    \end{center}
    \caption{The sets $\Gamma_i$ and $\Delta_i$ for $i \in \{1,2,3\}$.}
\end{figure}

\begin{definition}
We denote the three maximal segments parallel to the three edges of the triangle~$\calT_n$ by:
\begin{align*}
    L_1 &:= \{(1, n-1, 0), \dotsc, (1, 0, n-1)\}, \\
    L_2 &:= \{(n-1, 1, 0), \dotsc, (0, 1, n-1)\}, \\
    L_3 &:= \{(n-1, 0, 1), \dotsc, (0, n-1, 1)\}.
\end{align*}
\end{definition}

\begin{definition}
If $P$ and $Q$ are two points in $\calT'$, we denote by $H(P, Q)$ the maximal segment of~$\calT'$ containing~$P$ and~$Q$.
\end{definition}

\begin{lemma}
\label{lemma:extremes}
Let $P \in \Gamma_\ell$ and $Q \in \Delta_\ell$ for some $\ell \in \{1,2,3\}$.
If $H(P,Q) \neq L_i$ for all $i \in \{1,2,3\}$, then $P$ and $Q$ are the extremes of~$H(P,Q)$.
\end{lemma}

Let us describe the points of $H(P, Q)$ when
$P$ and $Q$ are its extremes.
In this case, the elements of~$H(P, Q)$ are the points $R_\lambda$ given by
\[
    R_\lambda = P + \frac{\lambda}{d}(Q-P), 
    \quad 
    \lambda = 0, \dotsc, d
\]
where $d$ is the greatest common divisor of the
entries of the vector $Q-P$. From this and \Cref{lemma:extremes}, we can easily describe the elements of $\mathcal{H}_{\mathrm{red}}$.

\begin{lemma}
\label{lemma:reduced_segments}
If $H\in \mathcal{H}_2$ is such that $g_H$ is squarefree,
then $H \in \mathcal{H}_{\mathrm{red}}$.
\end{lemma}

\begin{proof}
If $H \not \in \mathcal{H}_{\mathrm{red}}$, then the greatest common divisor of the
monomials of $g_H$ contains $x$, or~$y$, or~$z$ to a power greater than $1$,
so $g_H$ cannot be squarefree.
\end{proof}

Hence, to describe the polynomials~$g$ that are squarefree, it is enough to consider maximal segments in~$\mathcal{H}_{\mathrm{red}}$, which determine the triple $(a: b: c)$.
For instance, if a segment~$H$ has an extreme in $\Gamma_2$ and an extreme in~$\Delta_2$, one has six possible cases, see \Cref{table:cases}.
The other cases involved in $\mathcal{H}_{\mathrm{red}}$ can be treated similarly, and the resulting polynomials are essentially obtained by permuting the variables $x$, $y$ and $z$.

\begin{table}[ht]
\caption{All six possible cases when we take a segment~$H$ with a vertex in~$\Gamma_2$ and a vertex in~$\Delta_2$.
If $P_2 \in \Gamma_2$ and $Q_2\in \Delta_2$, let $d$ be the greatest common divisor of the entries of the vector $Q_2-P_2$.
We define the vector $(\alpha, \beta, \gamma)$ to be
\[
    (\alpha, -\beta, \gamma) := \frac{1}{d}(Q_2-P_2).
\]
The polynomials~$T$ appearing in the table are bivariate homogeneous polynomials of degree $d$; hence, each $T(y^{\beta}, x^{\alpha}z^{\gamma})$ admits a factorization of the form
\[
    T(y^{\beta}, x^{\alpha}z^{\gamma}) = \prod_{s=1}^d (\lambda_s y^\beta +\mu_s x^{\alpha}z^{\gamma}) \,.
\]
The cases (2), $i=n-2$; (5), $i=n-1$; (6), $i=n-2$
are omitted, since sub-cases of (1), $i=n-1$. Similarly, the cases (4), (5), (6), $i=1$ are also
omitted, since sub-cases of (3), $i=1$.
}
\begin{tabular}{cccrl}
    \toprule
    & $(P_2, Q_2)$ & $(a:b:c)$ & & $g_H$ \\
    \hline\\
    (1) &
    \begin{tabular}{c}
        \ $\left(\begin{smallmatrix}
        0 & n-1& 1 \\
        i & 0 & n-i
        \end{smallmatrix}\right)$,\\
        {\tiny $1\le i\le n-1$}
    \end{tabular}
    & $\left(\begin{smallmatrix} (n-i)(n-1) \\ i \\ -i(n-1) \end{smallmatrix}\right)$ &
    \begin{tabular}{r}
        {\small $i = 1$:}\\
        {\small $i = n-1$:}\\
        {\small $1<i<n-1$:}
    \end{tabular} &
    \begin{tabular}{l}
        $z \, (\lambda \, y^{n-1} + \mu \, x z^{n-2})$\\
        $z \, T(x,y)$ \\
        $z \, T(y^{\beta}, x^{\alpha}z^{\gamma})$
    \end{tabular}\\[24pt]
    (2) &
    \begin{tabular}{c}
        $\left(\begin{smallmatrix}
        0 & n-1& 1 \\
        i & 1 & n-i-1
        \end{smallmatrix}\right)$,\\
        {\tiny $1\le i\le n-1$}
    \end{tabular}
    & $\left(\begin{smallmatrix}
    (n-i-1)(n-1)-1 \\ -i \\ -i(n-1)
    \end{smallmatrix}\right)$&
    \begin{tabular}{r}
        {\small $i = 1$:}\\
        {\small $i = n-1$:}\\
        {\small $1<i<n-2$:}
    \end{tabular} &
    \begin{tabular}{l}
        $yz \, (\lambda \, y^{n-2} + \mu \, x\, z^{n-3})$\\
        $y \, (\lambda \, y^{n-2}z + \mu \, x^{n-1})$ \\
        $yz  \, T(y^{\beta}, x^{\alpha}z^{\gamma})$
    \end{tabular}\\[24pt]
    (3) &
    \begin{tabular}{c}
        $\left(\begin{smallmatrix}
        1 & n-1& 0 \\
        i & 0 & n-i
        \end{smallmatrix}\right)$,\\ 
        {\tiny $1\le i\le n-1$}
    \end{tabular}
    & $\left(\begin{smallmatrix}
    (n-i)(n-1) \\ -(n-i) \\ -i(n-1)
    \end{smallmatrix}\right)$&
    \begin{tabular}{r}
        {\small $i = 1$:}\\
        {\small $i = n-1$:}\\
        {\small $1<i<n-1$:}
    \end{tabular} &
    \begin{tabular}{l}
    $x\, T(y,z)$\\
    $x\, T(\lambda y^{n-1}+\mu x^{n-2}z)$\\
    $x \, T(y^{\beta}, x^{\alpha}z^{\gamma})$
    \end{tabular}\\[24pt]
    (4) &
    \begin{tabular}{c}
        $\left(\begin{smallmatrix}
        1 & n-1& 0 \\
        i & 1 & n-i-1
        \end{smallmatrix}\right)$,\\
        {\tiny $0\le i\le n-2$}
    \end{tabular}
    & $\left(\begin{smallmatrix}
    (n-i-1)(n-1) \\ -(n-i-1) \\ 1-i(n-1)
    \end{smallmatrix}\right)$&
    \begin{tabular}{r}
        {\small $i = 0$:}\\
        {\small $i = n-2$:}\\
        {\small $1<i<n-2$:}
    \end{tabular} &
    \begin{tabular}{l}
        $y \, (\lambda xy^{n-2}+ \mu z^{n-1})$\\
        $xy \, (\lambda y^{n-2}+\mu x^{n-3}z)$\\
        $xy\, T(y^{\beta}, x^{\alpha}z^{\gamma})$
    \end{tabular}\\[24pt]
    (5) &
    \begin{tabular}{c}
        $\left(\begin{smallmatrix}
        1 & n-2& 1 \\
        i & 0 & n-i
        \end{smallmatrix}\right)$,\\
        {\tiny $1\le i\le n-1$}
    \end{tabular}
    & $\left(\begin{smallmatrix}
    (n-i)(n-2) \\ 2i-n \\ -i(n-2)
    \end{smallmatrix}\right)$&
    \begin{tabular}{r}
        {\small $1<i<n-1$:}
    \end{tabular} &
    \begin{tabular}{l}
        $xz \, T(y^{\beta}, x^{\alpha}z^{\gamma})$
    \end{tabular}
    \\[24pt]
    (6) &
    \begin{tabular}{c}
        $\left(\begin{smallmatrix}
        1 & n-2& 1 \\
        i & 1 & n-i-1
        \end{smallmatrix}\right)$,\\
        {\tiny $0\le i\le n-1$}
    \end{tabular}
    & $\left(\begin{smallmatrix}
    (n-i-1)(n-2)-1 \\ 2i+1-n \\ 1-i(n-2)
    \end{smallmatrix}\right)$&
    \begin{tabular}{r}
        {\small $i = 0$:}\\
        {\small $i = n-1$:}\\
        {\small $1<i<n-2$:}
    \end{tabular} &
    \begin{tabular}{l}
        $yz\, (\lambda xy^{n-3}+\mu z^{n-2})$\\
        $xy\, (\lambda y^{n-3}z+\mu x^{n-2})$\\
        $xyz \, T(y^{\beta}, x^{\alpha}z^{\gamma})$
    \end{tabular}
    \\
    \bottomrule
\end{tabular}
\label{table:cases}
\end{table}

\newpage
The following theorem, which provides a classification of reduced free curves with a linear Jacobian syzygy of type \eqref{eq:3_blocks}, is an immediate consequence of the inspection of the zero loci of polynomials appearing in \Cref{table:cases}.

\begin{theorem}
\label{theorem:characterization_free_curves}
If a polynomial $g$ of degree $n\ge 4$ corresponds to a reduced free curve, not a set of concurrent lines, with a linear Jacobian syzygy of type \eqref{eq:3_blocks}, then the zero set of $g$ falls in one of the following four families:
\begin{itemize}
    \item union of a line with a set of concurrent lines;
    \item union of bitangent conics with at least one of their common tangent line;
    \item union of unicuspidal curves, belonging to a pencil, all with the same cusp and the same cuspidal tangent, with the common tangent line, and possibly a general line through the cusp;
    \item union of bicuspidal curves belonging to a pencil, all with the same two cusps, with at least one cuspidal tangent, and possibly the other tangent and/or the line passing through the two cusps.
\end{itemize}
Note that, besides lines, all the irreducible components of the zero set of~$g$ have the same degree.
\end{theorem}

We conclude this Section by computing the cardinality of~$\mathcal{H}_{\mathrm{red}}$.

\begin{lemma}
\label{lemma:num_max_seg}
The number of maximal segments $H(P, Q)\subseteq \mathcal{H}_2$ with $P\in \Gamma_\ell$ and $Q\in \Delta_\ell$ is $6n-11$ (here $\ell=1, 2, 3$).
\end{lemma}

\begin{proof}
We can consider the case $\ell=1$. The other two cases are symmetric.
We have that $\Delta_1 = D_a \cup D_b$, where
$D_a = \{ (0, j, n-j) \mid \ j = 1, \dotsc, n-1\}$ and
$D_b = \{(1, j, n-j-1) \mid j = 0, \dotsc, n-1 \}$.
Let $P_1 = (n-1, 1, 0)$, then
\[
    U_1 = \{H(P_1, Q) \mid Q\in D_a\}
\]
has $n-1$ maximal segments all contained in $\mathcal{H}_2$.
The set
\[
    U_2 = \{H(P_1, Q) \mid Q\in D_b\}
\]
has $n$ maximal segments but contains $\{(n-1, 1, 0), \dotsc, (1, n-1, 0)\}$, which
is not in $\mathcal{H}_2$, so $U_2 \cap \mathcal{H}_2$ has $n-1$ elements.
The sets $U_1$ and $U_2 \cap \mathcal{H}_2$ have the maximal segment~$L_2$ in common.
Overall, starting from the point $P_1$, 
we have therefore a set $X_1= U_1 \cap U_2 \cap \mathcal{H}_2$ with $2n-3$ distinct maximal segments.
We can do the same construction starting from $P_2 = (n-1, 0, 1)$ and we get a similar result, 
i.e., a set $X_2$ with $2n-3$ maximal segments contained
in $\mathcal{H}_2$ (here, in place of~$L_2$, we find~$L_3$).
Finally, if we take the point
$P_3 = (n-2, 1, 1)$, we get
that $\{H(P_3, Q) \mid Q\in D_a\}$ and $\{H(P_3, Q) \mid Q\in D_b\}$
again have $n-1$ and $n$ elements, respectively, but now the two sets have the two maximal
segments $L_2$ and $L_3$ in common.
Hence, the corresponding set $X_3$ of
maximal segments has again $2n-3$ elements.
Finally, it holds
\[
    X_1 \cap X_2 = \emptyset, \quad
    X_1 \cap X_3 = \{L_1\}, \quad
    X_2 \cap X_3 = \{L_2\}.
\]
Therefore, the number of maximal segments contained in $\mathcal{H}_2$ with one point
in $\Gamma_1$ and another point in $\Delta_1$ is $3(2n-3)-2 = 6n-11$.
\end{proof}

\begin{prop}
The number of elements of $\mathcal{H}_{\mathrm{red}}$ is $6(3n-8)$.
\end{prop}

\begin{proof}
Let $\mathcal{M}_{\ell}$ be the set of maximal segments $H(P, Q)$
with $P\in \Gamma_{\ell}$ and $Q\in \Delta_{\ell}$, for $\ell=1, 2, 3$.
From \Cref{lemma:num_max_seg} we know that each $\mathcal{M}_l$ has
$6n-11$ elements.
To compute the cardinality of~$\mathcal{H}_{\mathrm{red}} = \mathcal{M}_1 \cup \mathcal{M}_2 \cup \mathcal{M}_3$, we determine the common elements between~$\mathcal{M}_\ell$ and~$\mathcal{M}_m$.
Take, for instance, $\ell=1$ and $m=2$.
Then
\[
    \mathcal{M}_1 \cap \mathcal{M}_2 =
    \bigl\{ 
        H(P, Q) \in \mathcal{H}_2 \mid
        P \in \Gamma_1 \cap \Delta_2 
        \mbox{ and } 
        Q \in \Delta_1 \cap \Gamma_2
    \bigr\}.
\]
Let $P_1 = (n-1, 1, 0)$, $P_2 = (n-1, 0, 1)$, and $P_3 = (n-2, 1, 1)$
be the three points of~$\Gamma_1$ and let $P_1', P_2', P_3'$ be the
corresponding points of~$\Gamma_2$.
Note that $P_1, P_3, P_1', P_3'$ are aligned, so
$H(P_1, P_1') = H(P_1, P_3') = \cdots = H(P_3, P_3')$.
Hence, from these points we get a single element of~$\mathcal{M}_1 \cap \mathcal{M}_2$.
The other elements of this intersection are four, and are given by:
\[
    H(P_1, P_2'), \quad
    H(P_2, P_1'), \quad
    H(P_2, P_3'), \quad
    H(P_3, P_2').
\]
Note that $H(P_2, P_2')$ is not considered, since it is not in
$\mathcal{H}_2$.
From this computation, we see that the number of elements of
$\mathcal{H}_{\mathrm{red}}$ is $3(6n-11)-3\cdot 5 = 6(3n-8)$.
\end{proof}

\section{The case \texorpdfstring{$A=y$, $B=z$, $C=0$}{A=y, B=z, C=0}}
\label{special_case}

In the case $A=y$, $B=z$, and $C=0$, the existence of a linear Jacobian syzygy is sufficient to obtain
a classification of the corresponding curves, and the only two possibilities for the zero set of the polynomial~$g$ are the so called \emph{P{\l}oski curves} (see \cite{Cheltsov2017}), namely:
\begin{itemize}
    \item a union of conics, belonging to a hyperosculating pencil, if $n$ is even;
    \item a union of conics, belonging to a hyperosculating pencil, together with their common tangent line, if $n$ is odd.
\end{itemize}
Such curves are all free by \cite[Example 3.2 (2)]{BeorchiaMiroRoig2024},
and they are precisely the curves attaining both the local and global bounds on the \emph{Milnor number},
see \cite{Ploski2014} and \cite{Shin2016}.

The discussion in this case can be carried in the same spirit of \Cref{general_case}.
Since the result is already proved via other techniques in \cite[Proposition~4.4]{PatelRiedlTseng2024},
we only give a sketch of the proof for readability reasons.

We focus on the case when $n$ is even; the proof when $n$ is odd goes similarly.
Let
\[
    g = \sum_{i+j+k=n} g_{ijk} x^i y^j z^k
\]
be a homogeneous, degree $n$ polynomial in $x, y, z$.

To simplify the notations,
a point of $\calT_n$ (i.e., a monomial of degree $n$ in $x, y, z$) will
be denoted by $(i, j)$ instead of $(i, j, n-i-j)$ and the coefficient
$g_{ij\, n-i-j}$ of $g$ will be denoted by $g_{ij}$.
Let
\[
    v := y \partial_x g + z\partial_y g.
\]
The polynomial $v$ is homogeneous of degree $n$.
Let $c(i, j)$ denote the coefficient in~$v$ of the monomial~$x^iy^jz^{n-i-j}$.
Then
\begin{align}
\label{eq:coeff_Cij}
    c(i, j) &=
    \left\{
    \begin{array}{ll}
        (j+1)g_{ij+1}+(i+1)g_{i+1\, j-1} & \mbox{if $j \geq 1$ and $k \geq 1$}, \\
        (i+1)g_{i+1\, j-1} & \mbox{if $j \geq 1$ and $k = 0$}, \\
        g_{ij+1} & \mbox{if $j = 0$ and $k \geq 1$}.
    \end{array}
    \right.
\end{align}
We determine the conditions on the coefficients of~$g$ that make $v$ vanish.
Therefore, we solve the linear system in $g_{ij}$:
\begin{equation}
\label{eq:tot}
    c(i, j) = 0  \quad \mbox{for} \ (i, j) \in \calT_n \,.
\end{equation}
The following maximal segments of $\calT_n$:
\begin{align*}
    \mathcal{A}_j &= 
    \bigl\{
        (0, j) + t(1, -2) \mid t\in \mathbb{N}, \ t \leq j/2
    \bigr\},
    && \mbox{for} \ j = 0, \dotsc, n-1 \\
    \mathcal{B}_j &= 
    \bigl\{
        (n-j, j) +t(1, -2)\mid t\in \mathbb{N}, \ t \leq j/2
    \bigr\},
    && \mbox{for} \ j = 0, \dotsc, n
\end{align*}
form a partition of $\calT_n$, so \Cref{eq:tot} is equivalent to
\begin{align}
\label{eq:Aj}
    c(i, j) &= 0 \quad \mbox{for} \ (i, j) \in \mathcal{A}_j \text{ for } j = 0, \dotsc, n-1 \\
\label{eq:Bj}
    c(i, j) &= 0 \quad \mbox{for} \ (i, j) \in \mathcal{B}_j \text{ for } j = 0, \dotsc, n
\end{align}
Moreover, from~\eqref{eq:coeff_Cij}, it is possible to see that for every
coefficient $g_{h\ell}$ of $g$ there exists one and only one index $j$
such that $g_{h\ell}$ appears in \Cref{eq:Aj,eq:Bj}, so they can be solved independently.

As an example, consider the linear system $c(i, j) = 0$ for $(i, j) \in A_6$
(assuming $n$ is big enough). Since $A_6$ is given by
\[
    (0, 6), (1, 4), (2, 2), (3, 0) \,,
\]
from~\eqref{eq:coeff_Cij} we get the system of equations:
\[
    \left\{
    \begin{aligned}
        7g_{07} + \phantom{1}g_{15} &= 0\\
        5g_{15} + 2g_{23} &= 0 \\
        3g_{23} + 3g_{31} &= 0 \\
        \phantom{1}g_{31} \phantom{\ \ + 3g_{31}} &= 0
    \end{aligned}
    \right.
\]
The only solution of this system is the zero solution.

With a computation very similar to the one of the example above,
we see that all the maximal segments~$\mathcal{B}_j$ and all
the maximal segments $\mathcal{A}_j$ when $j$ is even give that the involved
coefficients of $g$ are zero, therefore the polynomial $g$ becomes:
\[
    \bar{g} = \sum_{j=1, 3, 5, \dots} h_j
    \quad \text{where} \quad
    h_j = \sum_{(i, j) \in \mathcal{A}_j}g_{ij}x^iy^jz^{n-i-j}
\]
and $g_{ij}$ satisfies \Cref{eq:Aj} for $j$ odd.

\noindent
Here are the first cases of \Cref{eq:Aj}:

\begin{table}[H]
\begin{tabular}{cl}
    \toprule
    $j = 1$ & $2g_{02}+g_{10} = 0$ \\
    $j = 3$ & $4g_{04}+g_{12} = 0, \ 2g_{12}+2g_{20} = 0$\\
    $j = 5$ & $6g_{06}+g_{14} = 0, \ 4g_{14}+2g_{22} = 0, \
    2g_{22}+3g_{30} = 0$\\
    $j = 7$ & $8g_{08}+g_{16} = 0,\ 6g_{16}+2g_{24} = 0, \ 4g_{24}+ 3g_{32} = 0,\
    2g_{32}+4g_{40} = 0$ \\ 
    \bottomrule
\end{tabular}
\end{table}
The solutions of these equations are (up to a scalar factor):
\begin{table}[H]
\begin{tabular}{cl}
    \toprule
    $j = 1$ & $g_{02} = 1, \ g_{10} = -2 $\\
    $j = 3$ & $g_{04} = 1, \ g_{12} = -4, \ g_{20} = 4$ \\
    $j = 5$ & $g_{06} = 1, \ g_{14} = -6, \ g_{22} = 12, \ g_{30} = -8$ \\
    $j = 7$ & $g_{08} = 1, \ g_{16} = -8, \ g_{24} = 24, \ g_{32} = -32, g_{40} = 16$\\
    \bottomrule
\end{tabular}
\end{table}
Therefore
\begin{align*}
    h_1 &= g_{02}z^{n-2}\left(y^2-2xz\right) \\
    h_3 &= g_{04}z^{n-4}\left(y^4-4xy^2z +4x^2z^2\right)\\
    &= g_{04}z^{n-4}\left(y^2-2xz\right)^2\\
    h_5 &= g_{06} z^{n-6}\left(y^6-6xy^4z+12x^2y^2z^2-8x^3z^3 \right)
    \\
    & = g_{06}z^{n-6}\left(y^2-2xz \right)^3 \\
    h_7 & = g_{08}z^{n-8}\left(y^8-8xy^6z+24x^2y^4z^2-32x^3y^2z^3-16x^4z^4\right)\\
    & = g_{08}z^{n-8}\left(y^2-2xz \right)^4
\end{align*}
and these expressions suggest the general form for $h_j$.
In conclusion, since $n = 2m$ is even, if we set $q = y^2-2xz$, the polynomial~$g$ becomes:
\[
    \bar{g} = \sum_{\ell=1}^m\lambda_{\ell}q^{\ell}z^{m-{\ell}}
\]
where $\lambda_{\ell} = g_{0\, 2{\ell}}$ are free parameters. Hence, $\bar{g}$ is a
product of conics.

\section{Applications and final comments}
\label{applications}

\subsection{Comparison with related results}

We compare our results with \cite[Theorem 3.5]{BuchweitzConca2013}, where a characterization of the syzygy matrices of free curves
with a Jacobian syzygy of the type $(ax,by,cz)$ is given.
The authors prove that if $abc\neq 0$, then the existence of the syzygy $(ax,by,cz)$ guarantees that the curve is free, and a syzygy matrix is given by
\[
    \begin{pmatrix}
        ax & \qquad \left(\frac{1}{c} - \frac{1}{b}\right)\ (n+2)^{-1} \partial_{yz}g \\
        by & \qquad \left(\frac{1}{a} - \frac{1}{c}\right)\ (n+2)^{-1}\partial_{xz}g\\
        cz & \qquad \left(\frac{1}{b} - \frac{1}{a}\right)\ (n+2)^{-1}\partial_{xy}g\\
    \end{pmatrix}.
\]

If $c=0$ and $ab\neq 0$, the authors prove that the curve is free if and only if $\partial_zf \in (x,y)$.

Our results describe precisely the form of the second column of the syzygy matrix, in particular with respect to the support of the polynomials appearing there.

\subsection{Free curves with a linear Jacobian syzygy and quasi-homogeneous singularities}
As an application of our results, we can characterize
the free curves with a linear Jacobian syzygy and only quasi-homogeneous singularities.
We recall the following definition.

\begin{definition}
An isolated singularity of a hypersurface is \emph{quasi-homogeneous} if and only if
there exists a holomorphic change of variables so that a local defining equation is weighted homogeneous.
Recall that $f(y_1, \dotsc, y_n) = \sum c_{i_1\cdots i_n}y_1^{i_1}\cdots y_n^{i_n}$ is said to be {\em weighted homogeneous}
if there exist rational numbers $\alpha_1,\cdots ,\alpha_n$ such that $\sum c_{i_1\cdots i_n}y_1^{i_1\alpha_1}\cdots y_n^{i_n\alpha _n}$ is homogeneous.
\end{definition}

It turns out that the quasi-homogeneity can be detected from a first syzygy matrix $M$ by 
\cite[Theorem 3.3]{AndradeBeorchiaMiroRoig2025}, for the case of free and nearly free curves, and
\cite[Theorem 4.8]{AndradeBeorchiaEtAl2025} for hypersurfaces with isolated singularities:

\begin{theorem}\label{thm: qh}
Let $C = V(g)$ be a reduced plane curve, let $M$ be a first syzygy matrix in a resolution of the Jacobian singular scheme, and let $p\in Sing(C)$.
Then, it holds
\[
    p \text{ is a quasi-homogeneous singularity }
    \iff
    \text{rk} \ M(p) \ge 1.
\]
\end{theorem}
 \begin{rmk}
   Some related results have been given also in \cite[Theorem 2.2]{Tohaneanu2013}, where the author proves that a plane curve with only isolated quasi-homogeneous singularities is free if and only if it admits a Jacobian syzygy given by a regular sequence. Moreover, Hilbert-Burch matrices with columns given by non regular sequences have been investigated in \cite{GamaLiraEtAl2025}.

 \end{rmk}
Let us apply the criterion of \Cref{thm: qh} to the case of free curves with a linear Jacobian syzygy.

A necessary condition for non quasi-homogeneity is clearly given by the linear dependence of the three linear entries of~$M$.
By the assumption of working with a Jordan normal form and by the discussion in \Cref{lemma:second_and_third_case}, 
this may occur only in the general $(ax,by,cz)$ case with either $b=0$ or $c=0$,
since we always assume $a\neq 0$, or in the last case $(y,z,0)$. It is well known, see \cite{Ploski2014},
that in the latter case all the curves have only one singular point, which is indeed a non quasi-homogeneous singularity if $n\ge 4$.

In the general case, we analyze the cases listed in  \Cref{table:cases}, and we see that $bc=0$ if and only if
\begin{enumerate}
    \item \label{first}
    $n$ is even and $i=\frac{n}{2}$ in case (5), or
    \item \label{second}
    $n$ is odd and $i=\frac{n-1}{2}$ in case (6).
\end{enumerate}
In case \eqref{first}, we have
\[
    d = \gcd\left(\frac{n-2}{2}, -(n-2), \frac{n-2}{2}\right) = \frac{n-2}{2},
    \quad 
    (\alpha, \beta, \gamma)=(1,2,1),
\]
\[
    H = 
    \bigl\{
        (1,n-2,1), (2, n-4,2), (3,n-6,3), \dotsc, (n/2, 0, n/2)
    \bigr\},
\]
and a first syzygy matrix is given by
\[
    M =
    \begin{pmatrix}
        \frac{n}{2}(n-2)x & 0\\
        0 & E\\
        -\frac{n}{2}(n-2)z & F
    \end{pmatrix} \,,
\]
where
\[
    E = 
    e_{0\,n-2\,0}y ^{n-2} + e_{1\, n-4\, 1}x y^{n-4}z+\cdots 
    + e_{\frac{n-2}{2}\, 0 \, \frac{n-2}{2}}x^{\frac{n-2}{2}}z^{\frac{n-2}{2}},
\]
and
\[
    F = 
    -(n-2)e_{0\, n-2\, 0}\ y^{n-3} z - \frac{n-4}{2}e_{1\, n-4\, 1} xy^{n-5} z^2 
    - \cdots -(n-2)e_{\frac{n-4}{2}\, 2 \, \frac{n-4}{2}}x^{\frac{n-4}{2}}y z^{\frac{n-2}{2}}.
\]
We see that $M(0,1,0)=0$ if and only if $e_{0\, n-2\, 0}=0$; in this case
we have $E=xz \widetilde{E}$ and $F=xz^2 \widetilde{F}$ for suitable polynomials $\widetilde{E}$ and $\widetilde{F}$.
The resulting polynomial $g$ would be of the type
\[
    n\, g = \det 
    \begin{pmatrix}
        x& \frac{n}{2}(n-2)x & 0\\
        y& 0                  & xz \widetilde{E}\\
        z& -\frac{n}{2}(n-2)z & xz^2 \widetilde{F}
    \end{pmatrix} 
    = n(n-2)x^2 z^2 \widetilde{E} - \frac{n}{2}(n-2)x^2yz^2 \widetilde{F},
\]
which gives a non reduced curve.

Similarly, in case \eqref{second},
a first syzygy matrix is given by
\[
    M =
    \begin{pmatrix}
        (\frac{n-1}{2}(n-2)-1)x & 0\\
        0 & E\\
        1-\frac{n-1}{2}(n-2)z & F\\
    \end{pmatrix}.
\]
As before, if $M(0,1,0)=0$, then the resulting polynomial $g$ is not square-free.

Thus, we conclude that all the singularities of the reduced free curves with a syzygy $(ax,by,cz)$ are always quasi-homogeneous.
We summarize our conclusions as follows.

\begin{theorem}
\label{theorem:ploski}
If a reduced free curve with a linear Jacobian syzygy has some non quasi-homogeneous singularity,
then it is either an even or an odd P{\l}oski curve.
\end{theorem}

We point out that the knowledge of the precise form of the Hilbert-Burch matrix is crucial for the previous result.

\subsection{Classification of nearly free curves with a linear Jacobian syzygy}
\label{section:nearly_free}

By a result of Du Plessis and Wall, see \cite[Proposition 1.1]{DuPlessisWall1999},
the reduced plane curves admitting a linear Jacobian syzygies are exactly those having a positive dimensional automorphism group.
The latter have been classified in \cite[Section 1]{AluffiFaber2000}.
Our results allows one to detect precisely the free curves.

Indeed, the minimal degree of a Jacobian syzygy is related to the \emph{total Tjurina number}~$\tau(g)$,
that is the degree of the zero-dimensional singular Jacobian scheme,
and we have the following result (see \cite[Theorem 1.2]{Dimca2017}):

\begin{lemma}
Let $C=V(g)$ be a reduced singular curve in $\p^2$ of degree $n \ge 3$ that is not a union of concurrent lines.
If $g$ admits a linear Jacobian syzygy, then
$n^2-3n+2 \le \tau(C)\le n^2-3n+3$.

Moreover, we have $\tau(C)= n^2-3n+3$, respectively $\tau(C)= n^2-3n+2$, if and only if $g$ is free, respectively is nearly free
--- that is, the module $\mathrm{Syz}(J_g)$ is generated by three syzygies of degrees $(1,n-1,n-1)$.
\end{lemma}

As a consequence, our results can be used to give a classification of nearly free curves with a linear Jacobian syzygy, which we now list.

\begin{corollary}
If a polynomial $g$ of degree $n\ge 4$ corresponds to a reduced nearly free curve with
a linear Jacobian syzygy, then up to a projectivity, such a syzygy is of the type $(ax,by,cz)$, and
the zero set of $g$ falls in one of the following families:
\begin{itemize}
    \item a union of bitangent conics belonging to a bitangent pencil, and possibly the secant line through the two bitangency points;
    \item a union of unicuspidal curves belonging to a pencil, all with the same cusp and the same cuspidal tangent, and possibly a general line through the cusp;
    \item a union of bicuspidal curves belonging to a pencil, all with the same two cusps, and possibly the line passing through the two cusps.
\end{itemize}
\end{corollary}

\begin{rmk}
By \cite[Theorem 3.5]{BuchweitzConca2013}, in the above cases we have $abc=0$;
this is indeed confirmed by the classification of nearly free conic-line arrangements with a linear Jacobian syzygy, given in \cite[Example 3.3]{BeorchiaMiroRoig2024}.
\end{rmk}

\subsection{Further investigation}
While our approach allows a complete description in the case of a linear Jacobian syzygy,
it seems hardly extendable to free curves with a higher degree first Jacobian syzygy, due to the lack of a simultaneous normal form.
However, our technique could be extended to higher dimensions and it deserves further investigation.

\bibliographystyle{alphaurl}
\bibliography{biblio}

\end{document}